\documentclass[11pt]{amsart}
\usepackage{mathpazo}
\usepackage{amssymb}
\usepackage{comment}
\usepackage{csquotes}
\usepackage{xcolor}
\usepackage{hyperref}
\hypersetup{
colorlinks=true,
linkcolor=blue,
filecolor=blue,
urlcolor=blue,
citecolor=blue,
pdfpagemode=Fullscreen
}
\usepackage[nodisplayskipstretch]{setspace}
\usepackage[margin=.7in]{geometry}

\newtheorem{theorem}{Theorem}[section]
\newtheorem{lemma}[theorem]{Lemma}
\theoremstyle{definition}

\newtheorem{definition}[theorem]{Definition}
\newtheorem{question}[theorem]{Question}
\newtheorem*{Acknowledgement}{\textnormal{\textbf{Acknowledgement}}}

\newtheorem{corollary}[theorem]{Corollary}
\newtheorem{proposition}[theorem]{Proposition}
\newtheorem{remark}[theorem]{Remark}

\numberwithin{equation}{section}
\newcommand{\beqa}{\begin{eqnarray*}}
	\newcommand{\eeqa}{\end{eqnarray*}}
\newcommand{\beqn}{\begin{eqnarray}}
	\newcommand{\eeqn}{\end{eqnarray}}

\newcommand{\R}{I\!\!R}

\newcommand{\I}{\mathcal{I}}
\newcommand{\N}{I\!\!N}

\newcommand{\e}{\varepsilon}
\newcommand{\del}{\delta}
\newcommand{\Del}{\Delta}

\newcounter{cnt1}
\newcounter{cnt2}
\newcounter{cnt3}
\newcommand{\blr}{\begin{list}{$($\roman{cnt1}$)$}
		{\usecounter{cnt1} \setlength{\topsep}{0pt}
			\setlength{\itemsep}{0pt}}}
	\newcommand{\bla}{\begin{list}{$($\alph{cnt2}$)$}
			{\usecounter{cnt2} \setlength{\topsep}{0pt}
				\setlength{\itemsep}{0pt}}}
		\newcommand{\bln}{\begin{list}{$($\arabic{cnt3}$)$}
				{\usecounter{cnt3} \setlength{\topsep}{0pt}
					\setlength{\itemsep}{0pt}}}
			\newcommand{\el}{\end{list}}
		\newtheorem{thm}{Theorem}

		\newtheorem{Def}[thm]{Definition}
		
		\newtheorem{rem}[thm]{Remark}
		\newcommand{\Rem}{\begin{rem} \rm}
			\newcommand{\bdfn}{\begin{Def} \rm}
				\newcommand{\edfn}{\end{Def}}

			\title{Semi denting points and related notions in Banach spaces}
			\author[ S. Basu, P. P. Behera, D. Gothwal, S. Seal ]
			{Sudeshna Basu$^{1}$, Priyanka Priyadarshini Behera$^{2}$, Deepak Gothwal $^{2}$, Susmita Seal$ ^{3}$     }
            
			\address{{$^{1}$} Sudeshna Basu, 
		Department of Mathematics and Statistics, 
				Loyola University, 
				Baltimore, MD 21210, USA }
			\email{sbasu1@loyola.edu, sudeshnamelody@gmail.com}

                \address {{$^{2}$} Priyanka Priyadarshini Behera and Deepak Gothwal, 
				Department of Mathematics,
				Indian Institute of Technology Kharagpur, 
				Kharagpur, India-721302}
			\email{bpriyadarshini18@gmail.com\\deepakgothwal190496@gmail.com}
			\address {{$^{3}$} Susmita Seal, 
				 School of Mathematical Sciences,
				National Institute of Science Education and Research, 
                Jatni, Khurda 752050,
				Odisha, India}
			\email{susmitaseal1996@gmail.com}

			\date{}
\begin{document}
\begin{abstract} 
%In this work, we investigate the stability of semi-denting points and their weaker variants under $\ell_p$-sums of Banach spaces for $1\leq p \leq \infty$. We also show that these notions can be lifted from ideals to their superspaces under suitable strong geometric assumptions. Stability results under projective tensor products have also been discussed. 
In this work, we study semi denting points and related notions in Banach spaces. We observe that $X$ has the Radon-Nikodým Property if and only if every closed bounded convex set has a semi denting point.
We also study the stability properties of semi denting, semi PC, and semi SCS points, as well as their $w^*$-analogues in Banach spaces, with respect to  $l_p$-sums ( $1\leq p  \leq \infty$), ideals, and projective tensor products.

\end{abstract}

\maketitle
%==================================================
\section{Introduction}
Let $X$ be a Banach space and $X^*$ its dual space. We denote by $B_X$ and $S_X$ the closed unit ball and the unit sphere of $X$, respectively. For $x \in X$ and $r >0$, $B[x,r]$ and $B(x,r)$ denote the closed ball and the open ball with center $x$ and radius $r>0$, respectively.
%Let $X$ denote a Banach space and $X^*$ denote the dual of $X$. $B_X$ and $S_X$ denote the unit ball and the unit sphere of $X$, respectively. %Suppose $X$ has two topologies $\tau_1$ and $\tau_2$. Let $x \in X$ has the property $P$ if  which also contains $x$. We say that $x$ has \emph{semi-P} if every open set in $\tau_1$ containing $x$ contains an open set of $\tau_2$. The condition of $\tau_2$ containing $x$ is relaxed in this ``semi''-notion.% 
%Let $X$ be a Banach space and $C$ be a bounded subset of $X$. 
For a bounded subset $C$ of $X$, we say that a point $x_0\in C$ is a {\it denting point} (resp. {\it Point of Continuity} (PC), {\it Small Combination of Slice point} (SCS point)) of $C$ if for every $\varepsilon > 0$, there exists a slice (resp. weakly open set, convex combination of slices) $S$ of $C$ such that $x_0 \in S \subset B(x_0,\varepsilon)$. Analogously, one can define their $w^*$-counterparts in dual spaces by considering $w^*$-slices, $w^*$-open sets and convex combinations of $w^*$-slices, respectively.

The {\it Radon–Nikodým Property} (RNP), a fundamental notion in Banach space theory, has a geometric characterization, namely, $X$ has the RNP if and only if every closed bounded convex subset of $X$ has a denting point. Following this geometric characterization, two weaker properties--{\it Convex Point of Continuity Property} (CPCP) and {\it Strong Regularity} (SR)--were introduced by replacing denting points by  PC and SCS points, respectively. RNP implies CPCP, which in turn implies SR, where none of the reverse implications holds. These properties are deeply connected to the extremal structure of closed bounded convex sets, the differentiability of convex functions, and the long-standing open problem: {\it does the Krein-Milman Property (KMP) imply RNP?} For more details, see \cite{ B1, B3, DU, GGMS, GMS} and references therein.

%The following definitions provide weaker variants of the above concepts.
%The following notions weaken these concepts.
\begin{definition} \cite{CL, BS4}
%Given a Banach space $X$ and a bounded subset $C$ of $X$,
%Let $X$ be a Banach space and $C$ be a bounded subset of $X$. 
We say that a point $x_0\in C$ is a {\it semi denting point} (resp. {\it semi Point of Continuity} (semi PC), {\it semi Small Combination of Slice point} (semi SCS point)) of $C$ if for every $\varepsilon > 0$, there exists a slice (resp. weakly open set, convex combination of slices) $S$ of $C$ such that $S \subset B(x_0,\varepsilon)$. Analogously, one can define their $w^*$-counterparts in the dual spaces by considering $w^*$-slices, $w^*$-open sets and convex combinations of $w^*$-slices, respectively. 
\end{definition}
It is well known that $X$ has RNP (resp. CPCP, SR) if and only if every closed bounded convex set has arbitrarily small slices (resp. weakly open sets, small combination of slices). See \cite{DU} and \cite{GGMS} for more details. 
\newpage

It is also well known that
\begin{theorem}\cite{DU, GGMS}
  \begin{enumerate}
     \item $X$ has the RNP if and only if every closed bounded convex set is the closed convex hull of its denting points.
\item $X$ has CPCP if and only if every closed bounded convex set is the weak closure of its PC.
\item $X$ has SR if and only if every closed bounded convex set is the norm closure of its SCS points.
\end{enumerate}
\end{theorem}

As an immediate consequence, we have the following result.

\begin{theorem}
\begin{enumerate}
\item[] 
\item $X$ has the RNP if and only if every closed bounded convex subset of $X$ has a semi denting point if and only if every closed bounded convex set is the closed convex hull of its semi denting points. 
\item $X$ has  CPCP if and only if every closed bounded convex subset of $X$ has a semi PC,  if and only if every closed bounded convex set is the weak closure of its semi PCs.
\item $X$ has  SR if and only if every closed bounded convex subset of $X$ has a semi SCS point,  if and only if every closed bounded convex set is the norm closure of its semi SCS points.
\end{enumerate}
\end{theorem} 
The above theorem strengthens our motivation to explore how far weaker local geometric structures can still detect global geometric properties of Banach spaces.

The study of these “semi” notions was initiated by Chen and Lin \cite{CL} in 1995  and has subsequently developed into an active area of research (see \cite{GJM}, \cite{MS}). Another significant feature of these weakened local geometric properties is their ability to precisely capture the geometry of the unit ball in the context of ball separation properties, where separation is achieved via closed balls rather than hyperplanes. Notably, this includes the classical Mazur Intersection Property (MIP) and Property (II) introduced by Chen and Lin in  1998 (\cite{CL1}).
In 2006, Giles  proved that $X$ has MIP if and only if all points of $S_{X^*}$ are semi $w^*$-denting points of $B_{X^*}$ (\cite{G}).
Later in 2024, Basu and Seal proved that $X$ has property (II) if and only if all points of $S_{X^*}$ are semi $w^*$-PC of $B_{X^*}$ \cite{BS4}. In 2021,  Bandyopadhyay, Ganesh, and  Gothwal introduced  the uniform notion of semi $w^*$-denting points and characterized the uniform version of MIP in terms of these points (\cite{BGG}).

%(One may refer to the mentioned reference for the definitions and notions involved). 

%Given a Banach space $X$ and a bounded subset $C$ of $X$,
 
These connections between the "semi'' notions and the geometric phenomenon point towards the fact that this "semi'' behavior is very crucial in understanding the geometry of Banach spaces. This motivates us to look further in this direction. 
%We essentially look into the stability of these notions. The stability of denting point, PC, and SCS of the closed unit ball  point with respect to  $l_p$ sum ( $1\leq p \leq \infty$) are well established in the literature (see \cite{W, S, BR}).  However, analogous stability results for semi denting points, semi PC and semi SCS points were open questions. 
%The aim of this work is to address these questions; our main contributions are summarized below.

The paper is organized as follows. Section \ref{prel} recalls the necessary notation and preliminary concepts. In Section \ref{stability}, we give the stability results for semi denting, semi PC and semi SCS points of the closed unit ball with respect to $l_p$-sum ( $1\leqslant p \leqslant \infty$). In Section \ref{ideal sec}, we study the stability of these "semi" notions and their weak star versions in the context of  Banach space ideals, specifically focusing on M-ideals, ai-ideals and strict-ideals. Finally, Section \ref{tensor} discusses their stability under projective tensor products.

\section{Notation and preliminaries} \label{prel}

For a nonempty subset $K$ of $X$, we consider the slice determined by $x^* \in X^*$ and $\delta >0$,
$$
S(K, x^*, \delta) := \{x \in K : x^*(x) > \mbox{sup}~ x^*(K) - \delta \}.
$$

For $K=B_X$, we assume that $x^* \in S_{X^*}$ for the sake of convenience.
We further denote $\sum_{i=1}^n \lambda_i S(K, x_i^*, \delta_i)$ to be a convex combination of slices of $K$, where $\lambda_i \geq 0, ~ \sum_{i=1}^n \lambda_i =1,~ 1 \leq i \leq n.$

\subsection{Ideals of Banach spaces} 
%We reiterate the following concepts and findings relevant to our discussion:
A closed subspace $Y \subset X$ is called an
ideal if there exists a norm one projection $P: X^* \rightarrow X^*$ such that $\text{Ker}(P) = Y^\perp$. Note that every Banach space $X$ is an ideal in its bidual $X^{**}$ , where the norm one
projection on $X^{***}$ is given by $x^{***}\rightarrow x^{***}|_X$. 

\begin{definition}
    A closed subspace $Y$ in a Banach space $X$ is called 
    \begin{enumerate}
        \item $M$-ideal if there exists a bounded linear projection $P: X^* \rightarrow X^*$ such that $\text{Ker}(P) = Y^\perp$ and $\| x^* \| = \| Px^* \| + \| x^* - Px^* \|$ for all $x^* \in X^*.$
        \item strict ideal if there exists a norm one projection $P: X^\ast \rightarrow X^\ast$ with $\text{ker}(P) = Y^\bot$ and $B_{P(X^\ast)}$ is $w^*$-dense in $B_{X^*}$.
        \item almost isometric ideal ($ai$-ideal) in $X$ if for every $\varepsilon >0$ and every finite dimensional subspace $F \subset X$ there exists $T: F \rightarrow Y$ such that $Tu = u $ for all $u \in F \bigcap Y$ and $\frac{1}{1+\varepsilon} \Vert u \Vert \leqslant \Vert Tu \Vert \leqslant (1+\varepsilon) \Vert u \Vert$ for all $u\in F$.
    \end{enumerate}
    \end{definition}

 It is well-known that if $Y$ is an $M$-ideal in $X$, then elements of $Y^\ast$ possess a unique norm-preserving extension to $X^\ast$ (\cite[Chapter I]{HWW}), leading to the identification $X^\ast = Y^\ast \oplus_1 Y^\perp$. 
% \begin{definition}
 %    An ideal $Y \subseteq X$ is classified as a strict ideal if, given a projection $P: X^\ast \rightarrow X^\ast$ with $\|P\| = 1$ and $\text{ker}(P) = Y^\bot$, it follows that $B_{P(X^\ast)}$ is $w^*$-dense in $B_{X^*}$ or in other words $B_{P(X^*)}$ is a norming set for $X.$ 
 %\end{definition}
% A quintessential illustration of a strict ideal is a Banach space $X$ under its canonical embedding in $X^{\ast\ast}.$
Recall that $f : Y^* \rightarrow X^*$ is called a Hahn Banach extension operator if $(fy^*)|_Y = y^*$ and $\Vert fy^* \Vert = \Vert y^* \Vert ,$ $\forall y^* \in Y^*$ and $y^* \in Y^*.$ 
%Following \cite{ALN}, we recall the following notion.
%\begin{definition}
%    A subspace $Y$ of $X$ is an almost isometric ideal ($ai$-ideal) in $X$ if for every $\varepsilon >0$ and every finite dimensional subspace $F \subset X$ there exists $T: F \rightarrow Y$ such that $Tu = u ,$ for all $u \in F \bigcap Y$ and $\frac{1}{1+\varepsilon} \Vert u \Vert \leqslant \Vert Tu \Vert \leqslant (1+\varepsilon) \Vert u \Vert,$ for all $u\in F$. 
%\end{definition}

%===================================================
\subsection{Projective tensor products}
Let $X$ and $Y$ be two Banach spaces. Then the projective norm $\|.\|_{\pi}$ on tensor product $X\otimes Y$ is given by
$$\|u\|_{\pi}= \inf \{ \sum_{i=1}^{n} \|x_i\|  \|y_i\| : u=\sum_{i=1}^{n} x_i \otimes y_i\}.  $$
%The projective norm $\|.\|_{\pi}$ on tensor product $X\otimes Y$ is given by  $\|u\|_{\pi}= \inf \{ \sum_{i=1}^{n} \|x_i\|  \|y_i\| : u=\sum_{i=1}^{n} x_i \otimes y_i\}. $
 The pair $(X\otimes Y, \|.\|_{\pi})$ is denoted by $X \otimes_{\pi} Y$ and its completion is known as the projective tensor product, which is denoted by $X \hat{\otimes}_{\pi} Y.$
  It follows that $B_{X \widehat{\otimes}_{\pi} Y} = \overline{\text{conv}} (B_{X} \otimes B_{Y}) = \overline{\text{conv}} (S_{X} \otimes S_{Y})$.

The projective tensor product $X \hat{\otimes}_{\pi} Y$ can also be formulated in terms of spaces of operators. It is straightforward to verify that the mapping $A\mapsto L_A$ (where $ \langle y, L_A(x) \rangle  =   A(x,y)$) is an isometric isomorphism between the spaces $(X\widehat{\otimes}_\pi Y)^*$ and $ L(X,Y^*)$. Thus, we have the identification $(X\widehat{\otimes}_\pi Y)^* = L(X,Y^*)$ (see \cite{R} for more details).
%  The space $ L(X,Y^*)$ is linearly isometric to the topological dual of $X\widehat{\otimes}_\pi Y$.
%Let $X$ and $Y$ be Banach spaces. Then we recall that the algebraic tensor product of $X$ and $Y$, denoted by $X \otimes_{\pi} Y$, endowed with the projective norm $\Vert \cdot \Vert_\pi$ given by $$\Vert u \Vert_{\pi} := \inf \bigg \{ \sum_{i=1}^n \Vert x_i \Vert \Vert \hat{x}_i \Vert : u=\sum_{i=1}^n x_i \otimes \hat{x}_i \bigg \}. 
%The pair $(X \otimes Y, \Vert \cdot \Vert_{\pi})$ is sometimes denoted by $X \otimes_{\pi} Y$ and its completion is called the projective tensor product, denoted by $X \widehat{\otimes}_{\pi} Y$. \\ For subsets $Y_1 \subset X$ and $Y_2 \subset Y$, we have $Y_1 \otimes Y_2 = \{ y_1 \otimes y_2 : y_1 \in Y_1, y_2\in Y_2 \}.$ We also know that $B_{X \widehat{\otimes}_{\pi} Y} = \overline{\text{conv}} (B_{X} \otimes B_{Y}) = \overline{\text{conv}} (S_{X} \otimes S_{Y})$ and $(X \widehat{\otimes}_{\pi} Y)^* = \mathcal{L}(X, Y^*) = \mathcal{B}(X \times Y)$ (see \cite{R} for more details).

\section{ $l_p$-Sums of Banach Spaces}\label{stability}

\subsection{Semi denting points}
%We first consider the case $1<p<\infty.$  Subsequently prove for $p=1$ and $p=\infty.$
%The proof of the following lemma follows from the arguments in \cite[Theorem 2.29 (c)]{L1}.

\begin{lemma}\label{tech1}
  Let $1 <p <\infty$ and $q$ be such that $1/p + 1/q =1$.
  %and $0<\varepsilon<\frac{8\|x\|}{3}$.  
  Let $z=(x, y) \in S_{X \oplus_p Y}$ with $x \neq 0, y\neq 0,$  and $0<\varepsilon<\frac{8\|x\|}{3}$ be such that 
  %$x, y\neq 0,$ 
  $
S(B_X, x^*, \delta) \subseteq B(\frac{x}{\Vert x \Vert}, \frac{\varepsilon}{4}),
$
and
$
S(B_Y, y^*, \delta) \subseteq B (\frac{y}{\Vert y \Vert}, \frac{\varepsilon}{4})
$, for some $x^*\in S_{X^*},$ $y^*\in S_{Y^*}$ and $\delta >0.$ Then $$S(B_{X \oplus_p Y}, z^*, \bigg (\frac{3\varepsilon}{8}\bigg )^p\delta)\bigcap (\|x\| S_X \times \|y\| S_Y) \subset B[z,\frac{\varepsilon}{4}],$$ where $z^*=(\|x\|^{\frac{p}{q}}x^*,\|y\|^{\frac{p}{q}}y^*)$.
\end{lemma}

\begin{proof}
    Let $w=(x_1,y_1) \in B_{X \oplus_p Y}$ be such that $\|x_1\|=\|x\|$ and $\|y_1\|=\|y\|$.
    
    Suppose that $\|z-w\|_p > \frac{\varepsilon}{4}$.

%Case 1: $\|x_1\|=\|x\|$ and $\|y_1\|=\|y\|$.

Let $a_1=\|x_1-x\|$, $a_2=\|y_1-y\|$, $b_1=\|x_1\|=\|x\|$ and $b_2=\|y_1\|=\|y\|$.
We then have, $$a_1=\|x_1-x\| \leq \|x_1\|+\|x\|=2b_1 \quad \text{and} \quad a_2=\|y_1-y\| \leq \|y_1\|+\|y\|=2b_2.$$
If $\|x_1-x\| \leq \frac{\e}{4}\|x_1\|$ and $\|y_1-y\| \leq \frac{\e}{4} \|y_1\|$, then $\|z-w\| \leq \frac{\e}{4}$, which is a contradiction. 
Hence, $\|x_1-x\| >\frac{\e}{4}\|x_1\|$ or $\|y_1-y\| >\frac{\e}{4}\|y_1\|$. 

Without loss of generality, let $\|x_1-x\| >\frac{\varepsilon}{4}\|x_1\|$, i.e, $\frac{a_1}{b_1} >\frac{\e}{4}.$ \\
We claim that
$x^*(x_1) \leq (1-\delta)b_1$ and $y^*(y_1) \leq b_2$. 
Indeed, if $x^*(x_1) >(1-\delta)b_1$, then $x^*(\frac{x_1}{\|x_1\|}) >1-\delta$, which implies $\|\frac{x_1}{\|x_1\|}-\frac{x}{\|x\|}\| <\frac{\e}{4} <\frac{a_1}{b_1}$. This results in $\|x_1-x\| <\|x_1-x\|$, which leads to a contradiction. Also, $\|y_1\|=b_2$. So, $y^*(y_1) \leq b_2$. Hence, our claim follows. Thus

 \beqa 
z^*(w) & = &\|x\|^{\frac{p}{q}}x^*(x_1) + \|y\|^{\frac{p}{q}}y^*(y_1)
\\ & \leq & b_1^{\frac{p}{q}} (1-\delta)b_1 + b_2^{\frac{p}{q}}b_2
\\ & = & b_1^p   (1-\delta) + b_2^p 
\\ & = & 1- b_1^p \delta
%\\ & = & 1-\alpha_0\delta 
\\ & \leq & 1- \bigg (\frac{3\e}{8} \bigg )^p \delta. 
\eeqa 
So, $z^*(w) > 1-(\frac{3\e}{8})^p \delta$ implies that $\|w-z\|_p\leqslant \frac{\varepsilon}{4}$.
\end{proof}

\begin{theorem} \label{thmlp} 
For $1 <p <\infty$, $(x, y) \in S_{X \oplus_p Y}$ is a semi denting point of $B_{X \oplus_p Y}$ if and only if  $\frac{x}{\Vert x \Vert}$ is a semi denting point of $B_X$ whenever $x \neq 0$ and $\frac{y}{\Vert y \Vert}$ is a semi denting point of $B_Y$ whenever $y \neq 0$.
%all the following conditions are satisfied. \begin{enumerate}     \item $\frac{x}{\Vert x \Vert}$ is a semi denting point of $B_X$ whenever $x \neq 0$.     \item $\frac{y}{\Vert y \Vert}$ is a semi denting point of $B_Y$ whenever $y \neq 0$. \end{enumerate}
%$\frac{x}{\Vert x \Vert}$ is a semi denting point of $B_X$ whenever $x \neq 0$ and  $\frac{y}{\Vert y \Vert}$ is a semi denting point of $B_Y$ whenever $y \neq 0$.
\end{theorem}

\begin{proof}
Let $q >0$ be such that $1/p + 1/q = 1$. Let $(x, y)$ be a semi denting point of $B_{X \oplus_p Y}$. 
%The case when $x=0$ or $y=0$ follows similarly as  in Theorem \ref{l1}
We will prove that $\frac{x}{\Vert x \Vert}$ is a semi denting point of $B_X$ whenever $x \neq 0$. The other case follows similarly. 
So, let $x \neq 0$ and $0<\varepsilon <\|x\|$.
 Then there exist $\delta >0$ and $(x^*,y^*) \in S_{X^* \oplus_q Y^*}$ such that
$$ 
S(B_{X \oplus_p Y}, (x^*, y^*), \delta) \subset B((x, y), \varepsilon). 
$$
%where  $(x^*, y^*) \in S_{X^* \oplus_q Y^*}$, with $\frac{1}{p} + \frac{1}{q} =1$. \\
We claim that $x^* \neq 0.$ 
%and $y^* \neq 0$. 
Suppose, on the contrary, that $x^* =0$.
%or $y^*=0$. Without loss of generality, let $y^*=0$. 
Then $y^* \in S_{Y^*}$. Let us now choose $\hat{y} \in S(B_{Y}, y^*, \delta)$. Then,
\begin{align*}
 (0,\hat{y}) &  \in S(B_{X \oplus_p Y}, (x^*, y^*), \delta) \subset B((x, y), \varepsilon).
\end{align*}
So, $$ \varepsilon^p > \Vert (0, \hat{y})-(x,y) \Vert_p^p = \Vert x \Vert^p + \Vert \hat{y} - y \Vert^p  \geq \Vert x \Vert^p, $$
which leads to a contradiction. Hence, our claim follows.\\
%Let us now consider the slice $S(B_{X}, \frac{x^*}{\Vert x^* \Vert}, \frac{\delta}{2})$ 
%and $S(B_Y, \frac{y^*}{\Vert y^* \Vert}, \frac{\delta}{2})$ 
%of $B_{X}.$ 
%and $B_Y$, respectively. 
Let us choose $(\hat{x}, \hat{y}) \in S(B_{X \oplus_p Y}, (x^*, y^*), \frac{\delta}{4}) \subset B((x, y), \varepsilon)$. 
Let $x_0 \in S(B_{X}, \frac{x^*}{\Vert x^* \Vert}, \frac{\delta}{2})$ and $y_0\in B_Y$ be such that $y^*(y_0)\geqslant \|y^*\|(1-\frac{\delta}{2}).$ 
%$y_0 \in S(B_Y, \frac{y^*}{\Vert y^* \Vert}, \frac{\delta}{2})$. 
Then, we obtain that $( \Vert \hat{x} \Vert x_0 , \Vert \hat{y} \Vert y_0 ) \in S(B_{X \oplus_p Y}, (x^*, y^*), \delta)$. Indeed, 
\begin{align*}
 \Vert \hat{x} \Vert x^*(x_0) +  \Vert \hat{y} \Vert y^*(y_0)  & >   (\Vert \hat{x} \Vert \Vert x^* \Vert + \Vert \hat{y} \Vert \Vert y^* \Vert) (1-\frac{\delta}{2})\\
& \geqslant (x^*, y^*) (\hat{x}, \hat{y}) (1-\frac{\delta}{2})\\
& > (1-\frac{\delta}{4})(1-\frac{\delta}{2})\\
& > 1-\delta.
\end{align*}
So, $( \Vert \hat{x} \Vert x_0 , \Vert \hat{y} \Vert y_0 ) \in B((x, y), \varepsilon)$. This implies that 
\begin{align*}
& \Vert ( \Vert \hat{x} \Vert x_0 , \Vert \hat{y} \Vert y_0 ) - (x, y) \Vert_p < \varepsilon \\
& \hspace{-0.3cm} \Rightarrow \Vert \Vert \hat{x} \Vert x_0 - x \Vert^p + \Vert \Vert \hat{y} \Vert y_0 - y \Vert^p < \varepsilon^p \\
& \hspace{-0.3cm} \Rightarrow  \Vert \Vert \hat{x} \Vert x_0 - x \Vert < \varepsilon ~ \text{and} ~ \Vert \Vert \hat{y} \Vert y_0 - y \Vert  < \varepsilon .
\end{align*}
Now, 
\begin{align*}
\Vert \Vert x \Vert x_0 - x \Vert & \leq \Vert \Vert \hat{x} \Vert x_0 - x \Vert + \bigg  \vert \Vert \hat{x} \Vert - \Vert x \Vert \bigg \vert  \Vert x_0 \Vert \\
& < \varepsilon + \Vert \hat{x} - x \Vert < \varepsilon + \varepsilon = 2 \varepsilon.
\end{align*}
%Similarly, $\Vert \Vert y \Vert y_0 - y \Vert < 2 \varepsilon.$
Finally, 
\begin{align*}
\bigg \Vert x_0 - \frac{x}{\Vert x \Vert} \bigg \Vert & = \frac{\Vert \Vert x \Vert x_0 - x \Vert}{\Vert x \Vert} < \frac{2 \varepsilon}{\Vert x \Vert}.
\end{align*}
%In a similarly way we also have, $\bigg \Vert y_0 - \frac{y}{\Vert y \Vert} \bigg \Vert < \frac{2 \varepsilon}{\Vert y \Vert}.$
This proves that $S \left(B_{X}, \frac{x^*}{\Vert x^* \Vert}, \frac{\delta}{2} \right ) \subset B \left (\frac{x}{\Vert x \Vert}, \frac{2 \varepsilon }{\Vert x \Vert} \right ).$ 
%and $S \bigg (B_{Y}, \frac{y^*}{\Vert y^* \Vert}, \frac{\delta}{2} \bigg ) \subset B \bigg (\frac{y}{\Vert y \Vert}, \frac{2 \varepsilon }{\Vert y \Vert} \bigg )$.
%Since $\varepsilon$ is independent of the choices of $x$ and $y$, so, 
Hence, $\frac{x}{\Vert x \Vert} $ is semi denting point of $B_X.$ 
%and $\frac{y}{\Vert y \Vert} \in $ semi dent  $(B_Y)$, respectively. 

Conversely, suppose that
%conditions (i) and (ii) hold.
$\frac{x}{\Vert x \Vert}$ is a semi denting point of $B_X$ whenever $x \neq 0$ and $\frac{y}{\Vert y \Vert}$ is a semi denting point of $B_Y$ whenever $y \neq 0$.

%if $x \in$ semi dent $B_X$ and $y=0$ or $x=0$ and $y \in $semi dent $B_Y$, then $(x,y) \in$ semi dent $B_{X \oplus_p Y}$. This follows from the arguments in the $\ell_1$ case.
Case 1: Either $x=0$ or $y=0.$

Without loss of generality, let $y=0$. Then $x\in S_X$ is a semi denting point of $B_X$. Let $\varepsilon >0.$ So, there exists a slice $S(B_{X}, x^*, \alpha)$ of $B_{X}$ with $x^*\in S_{X^*}$ and $\alpha>0$ with $(1-(1-\alpha)^p)^{1/p}<\varepsilon$ such that $S(B_{X}, x^*, \alpha) \subset B(x, \varepsilon)$. Then, for $q$ satisfying $1/p + 1/q = 1$, we have $z^*:=(x^*, 0) \in S_{X^* \oplus_q Y^*}$ and 
%consider  slice $S(B_{X \oplus_p Y}, z^*, \alpha)$ of $B_{X \oplus_1 Y}$ such that
\begin{align*}
S(B_{X \oplus_p Y}, z^*, \alpha) & \subset S(B_{X}, x^*, \alpha) \times \varepsilon B_{Y} \\
 & \subset B(x, \varepsilon) \times \varepsilon B_{Y} \\
 & \subset B((x, 0), 2 \varepsilon).
\end{align*}
 This shows that $(x,0)$ is a  semi denting point of $B_{X \oplus_p Y}.$ 

%Our first claim is $x^*\neq 0.$ If not, then $\|y^*\|=1$ and we choose $y_0\in S_Y$ such that $y^*(y_0)>1\delta.$ Then $(0,y_0)\in $

Case 2: Both $x$ and $y$ are nonzero.
%So, let $x, y \neq 0$ and 

Then $\frac{x}{\|x\|} $ is a semi denting point of $B_X$ and $\frac{y}{\|y\|}$ is a semi denting point of $B_Y$. We need to show that $z:=(x,y)$ is a semi denting point of $B_{X \oplus_p Y}$. 
Let $0 <\e < \min\{\frac{8\|x\|}{3}, 2\}$. Then  there exist $x^* \in S_{X^*}$, $y^* \in S_{Y^*}$ and $\delta >0$
%$\del(x,\e),\del(y,\e)>0$ 
%$\del(\e)=\min\{\del(x,\e),\del(y,\e)\} >0$ 
such that 
$$
S(B_X, x^*, \delta) \subseteq B(\frac{x}{\Vert x \Vert}, \frac{\varepsilon}{4}),
$$
and
$$
S(B_Y, y^*, \delta) \subseteq B (\frac{y}{\Vert y \Vert}, \frac{\varepsilon}{4}).
$$
%Choose $x^*$ and $y^*$ and $\delta$ as above for $\e/4$.
Let $z_1^*=\|x\|^{\frac{p}{q}}x^*$, $z_2^*=\|y\|^{\frac{p}{q}}y^*$ and $z^*=(z_1^*,z_2^*)$.
%Hence from Lemma $\ref{tech1},$ we have 
%$$S(B_{X \oplus_p Y}, z^*, (\frac{3\varepsilon}{8})^p\delta)\bigcap (\|x\| S_X \times \|y\| S_Y) \subset %B[z,\frac{\varepsilon}{4}].$$

%Now, we go to the general case. 
Let $\Delta_p(\e)$ be the modulus of convexity of $\ell_p^2$. Since $\ell_p^2$ is uniformly convex, $\Delta_p(\e) >0$ for each $\e >0$. Choose $\alpha$ such that $0 <\alpha <\frac{\e}{2}$ and $\Delta_p(\frac{\alpha}{2})+\alpha <(\frac{3\e}{16})\delta$.\\
Let $w=(x_1,x_2) \in B_{X\oplus_p Y}$ be such that $z^*(w) > 1 - \Del_p(\frac{\alpha}{2})$. Then,
\[
 \|x\|^{\frac{p}{q}} \|x_1\| + \|y\|^{\frac{p}{q}}\|x_2\|  >\|x\|^{\frac{p}{q}}x^*(x_1)+\|y\|^{\frac{p}{q}}y^*(x_2) > 1 - \Del_p(\frac{\alpha}{2}).
\]
Let $\tilde{x} := (\|x\|)$ and $\tilde{y}:=(\|y\|)$. We then have, $\tilde{z}=(\tilde{x}, \tilde{y}) \in S_{\ell_p^2}$, $\tilde{w} := (\|x_1\|,\|x_2\|) \in B_{\ell_p^2}$ and $u := (\|x\|^{\frac{p}{q}},\|y\|^{\frac{p}{q}}) \in S_{\ell_q^2}$. Then,
\[
\tilde{z}(u) = \|x\| \|x\|^{\frac{p}{q}} + \|y\| \|y\|^{\frac{p}{q}}= 1 \quad \mbox{and} \quad \tilde{w}(u) > 1 - \Del_p(\frac{\alpha}{2}).
\]
%Here, we may consider $u$ acting as a functional on $\ell_p^2$. 
Also, $\tilde{w}(u) > 1 - \Del_p(\frac{\alpha}{2})$, which further implies that $\frac{\tilde{z}(u)+\tilde{w}(u)}{2}
>1-\Del_p(\frac{\alpha}{2})$, i.e., $\frac{\|\tilde{z}+\tilde{w}\|_p}{2} >1-\Del_p(\frac{\alpha}{2})$. So, by the definition of the modulus of convexity, we have
\[
\|\tilde{z}-\tilde{w}\|_p = \left( \bigg|\|x_1\| - \|x\|\bigg|^p +  \bigg|\|x_2\| - \|y\|\bigg|^p\right)^{\frac{1}{p}}< \alpha.
\]
Let $c=(c_1,c_2)$, where $c_1 = \frac{x_1\|x\|}{\|x_1\|}$ and $c_2=\frac{x_2\|x\|}{\|x_2\|}$. Clearly, $\|c_1\|=\|x\|$ and $\|c_2\|=\|y\|$. And
\[
\|w-c\|_p= \left(\bigg|\|x_1\| - \|x\|\bigg|^p + \bigg|\|x_2\| - \|y\|\bigg|^p\right)^{\frac{1}{p}}< \alpha.
\]
Since $z^*(w) >1-\Del_p(\frac{\alpha}{2})$, we have 
\begin{align*}
    z^*(c) \geq z^*(w) -\|w-c\|_p > 1-\Del_p(\frac{\alpha}{2}) - \alpha >1 - (\frac{3\e}{16})\delta.
\end{align*}
%, by the choice of $\alpha$. 
%By the first case,
Thus, from Lemma $\ref{tech1},$ we have 
\[
\|z-c\|_p <\frac{\e}{2}.
\]
Finally,
\[
\|z-w\|_p \leq \|z-c\|_p + \|c-w\|_p <\e.
\]
%which shows that $\frac{x}{\Vert x \Vert}$ is a semi denting point of $B_X$.
%In a similar way, we can also show that  
\end{proof}

The result in the previous theorem can be generalized to a countable family of Banach spaces. %For that, we need the following notation.

%For $1 < p <\infty$. For a collection $\{X_i\}_{i=1}^\infty$, define $X_p:=\bigoplus_p X_i=\{(x_i)_{i=1}^\infty: \sum_{i=1}^\infty \|x_i\|^p <\infty\}$ with $\|(x_i)\|_p=\left(\sum_{i=1}^\infty \|x_i\|^p\right)^{1/p}$.

\begin{theorem}Let $\{X_i : i\in I\}$ be a family of Banach spaces and
$X = \bigoplus_p X_i, ~ 1<p<\infty$. Then $x=(x_i)_{i=1}^\infty \in S_{X}$ is a semi denting point of $B_{X}$ if and only if  $\frac{x_i}{\|x_i\|}$ is a semi denting point of $B_{X_i}$ whenever $x_i \neq 0.$ 
\end{theorem}

\begin{proof}
Let $x$ be a semi denting point of $B_{X}$ and $i_0 \in \N$ be such that $x_{i_0} \neq 0$. Also, let $Z=\mathop{\bigoplus_p}\limits_{i \neq i_0} X_i$.  
%$X$ is isometrically isomorphic to $X_{i_0} \oplus Z$. 
Then $X= X_{i_0} \oplus_p Z.$ Thus,
by Theorem ~\ref{thmlp}, $\frac {x_{i_0}}{\|x_{i_0}\|} $ is a semi denting point of $B_{X_{i_0}}$.

Conversely, suppose that $\frac {x_i}{\|x_i\|}$ is a semi denting point of $B_{X_i}$ for each $i$ such that $x_i \neq 0$ and $\e >0$. 
%Since $(x_i)_{i=1}^{\infty} \in S_{X}$, $\sum\limits_{i=1}^\infty \|x_i\|^p = 1$. 
Choose $n_0 \in \N$ such that $\sum\limits_{i >n_0}\|x_i\|^p < (\frac{\e}{8})^p$. So, $\sum_{i \leq n_0}\|x_i\|^p >1-(\frac{\e}{8})^p >(1-\frac{\e}{8})^p$.

Consider $W:=\mathop{\bigoplus_p}\limits_{i=1}^{n_0} X_i$. 
%{ \bf This notation needs to change }
Let 
%$z=(x_i)_{i=1}^{\infty}$, 
$w=(x_i)_{i=1}^{n_0}$ and $\tilde{w}=\frac{w}{\|w\|_p}$. Note that Theorem ~\ref{thmlp} can be inductively extended to a finite sum of Banach spaces. So, it follows that $\tilde{w}$ is a semi denting point of $B_W$. Therefore, there exist $0< \del <\frac{\e}{8}$ and $w^* \in S_{W^*}$ such that
\[
S(B_W,w^*, \del) \subseteq B(\tilde{w},\frac{\e}{8}).
\]
Let $z^*=(z^*_i)_{i=1}^{\infty}$ be defined as 
\[
z^*_i =
\begin{cases}
w^*_i &\text{if } i \leq n_0, \\
 0 &\text{otherwise}.
 \end{cases}
 \]
 Clearly, $z^* \in S_{X^*}$. 
 Let $y:=(y_i)_{i=1}^\infty \in B_X$ and $v=(y_i)_{i=1}^{n_0}$. Let $z^*(y) >(1-\del^p)^{\frac{1}{p}}$. So, $w^*(v) >(1-\del^p)^{\frac{1}{p}}>1-\del$, which further gives $\|v\|_p^p >1-\del^p$. i.e., $\sum\limits_{i \leq n_0}\|y_i\|^p >1-\del^p$. Therefore, $\sum\limits_{i >n_0}\|y_i\|^p <\del^p$ and hence, $\left(\sum\limits_{i>n_0} \|y_i\|^p\right)^{\frac{1}{p}} <\del <\frac{\e}{8}$.
Also, $\|v-\tilde{w}\|_p <\frac{\e}{8}$. Hence, 
\begin{align*}
    \|w-v\|_p & \leq \|w-\tilde{w}\|_p + \|\tilde{w}-v\|_p \\
    & \leq \frac{\e}{8} + \frac{\e}{8} =\frac{\e}{4}.
\end{align*}
Finally,
\beqa
 \|y-x\|_p &=& \left(\sum_{i=1}^\infty\|y_i-x_i\|^p\right)^{\frac{1}{p}}\\
 &\leq &\left(\sum_{i \leq n_0}\|y_i-x_i\|^p\right)^{\frac{1}{p}} + \left(\sum_{i >n_0} \|y_i-x_i\|^p\right)^{\frac{1}{p}}\\
 &\leq& \|v-w\|_p +  \left(\sum_{i >n_0} \|y_i\|^p\right)^{\frac{1}{p}} + \left(\sum_{i >n_0} \|x_i\|^p\right)^{\frac{1}{p}}\\
 &<& \frac{\e}{4} + \frac{\e}{8} + \frac{\e}{8}\\
 &<& \e.
\eeqa
\end{proof}

%\begin{corollary}
 %   Let $(X_i)_i$ be a sequence of Banach spaces and $(x_i)_i \in S_{\oplus_{\infty}X_i}$ be a semi SCS point of $B_{\oplus_{\infty}X_i}$, then $x_i \in S_{X_i}$ is a semi SCS point of $B_{X_i}$, for all $i\in \mathbb{N}$.
%\end{corollary}

\begin{theorem} \label{l1} 
$(x,y) \in S_{X \oplus_1 Y}$ is a semi  denting point of $B_{X \bigoplus_1 Y}$ if and only if either $x$ is a semi denting point of $B_X$ and $y=0$ or $y$ is a semi denting point of $B_Y$ and $x=0$.
%one of the following holds. \begin{enumerate} \item $x$ is a semi denting point of $B_X$ and $y=0$. \item $y$ is a semi denting point of $B_Y$ and $x=0$. \end{enumerate}
%Then the following conditions hold.
%\begin{enumerate}
 %   \item If $x \in S_X$ is semi denting point of $B_X$, then $(x,0)$ is a semi denting point of $B_{X \oplus_p Y},$ for $1 \leq p < \infty$.

    %\item If $(x,0) \in S_{X \oplus_p Y}$ is a semi denting point of $_{X \oplus_p Y}$, for $1<p<\infty$, then $x$ is a semi denting point of $B_X$.
    
%\end{enumerate}
\end{theorem}

\begin{proof}
   Let $(x,y) \in S_{X \oplus_1 Y}$ be a semi denting point of $B_{X \bigoplus_1 Y}$. If possible, let  $x, y \neq 0$. Choose $\varepsilon >0$ such that $\varepsilon <\min\{\|x\|,\|y\|\}$. Then, there exists a slice $S(B_{X \oplus_1 Y}, (x^*, y^*), \delta)$ of $B_{X \oplus_1 Y}$ such that 
$$
S(B_{X \oplus_1 Y}, (x^*, y^*), \delta) \subset B((x, y), \varepsilon),
$$
 where $(x^*, y^*) \in S_{X^* \oplus_\infty Y^*}$ and $\delta >0$. Then either $\|x^*\|=1$ or $\|y^*\|=1$.
Without loss of generality, let $\Vert x^* \Vert =1$ and choose $x_0 \in S(B_{X}, x^*, \delta).$ Then
$(x_0, 0) \in S(B_{X \oplus_1 Y}, (x^*, y^*), \delta) \subset B((x, y), \varepsilon).$
This implies that $$\varepsilon > \Vert (x_0, 0)-(x,y) \Vert = \Vert x_0 - x \Vert + \Vert y \Vert \geq \Vert y \Vert,$$ which leads to a contradiction. Similarly, we get a contradiction if $\Vert y^* \Vert =1$. Hence $x,y$  both cannot be nonzero simultaneously. 
 Without loss of generality, let $y=0$. Let $0<\varepsilon <1$. Since $(x,0)$ is semi denting point of $B_{X \bigoplus_1 Y}$, there exists a slice $S(B_{X \oplus_1 Y}, (x^*, y^*), \delta)$ of $B_{X \oplus_1 Y}$ such that $$S(B_{X \oplus_1 Y}, (x^*, y^*), \delta) \subset B((x, 0), \varepsilon),$$ where $(x^*, y^*) \in S_{X^* \oplus_\infty Y^*}$ and $\delta >0$. We claim that $\|x^*\|=1.$ If not, then $\|y^*\|=1$ and we choose $y_0\in S_Y$ such that $y^*(y_0)>1-\delta.$ Then $$(0,y_0)\in S(B_{X \oplus_1 Y}, (x^*, y^*), \delta) \subset B((x, 0), \varepsilon).$$
Thus $1<\|x\|+\|y_0\|= \|(0,y_0)-(x,0)\|<\varepsilon,$ a contradiction.
%We now consider the two cases. \\
%Case 1: 
Hence, $\|x^*\|=1$. Now, let $x_0 \in S(B_{X}, x^*, \delta)$. Then 
\begin{align*}
& (x_0, 0) \in S(B_{X \oplus_1 Y}, (x^*, y^*), \delta) \subset B((x, 0), \varepsilon) \\
& \hspace{-0.3cm} \Rightarrow \Vert (x_0, 0) - (x, 0) \Vert < \varepsilon \\
& \hspace{-0.3cm} \Rightarrow \Vert x_0 - x \Vert < \varepsilon.
\end{align*}
%which proves that $x_0 \in B(x, \varepsilon)$. 
Hence, $S(B_{X}, x^*, \delta) \subset B(x, \varepsilon)$. Thus $x$ is a semi denting point of $B_{X}$.

Conversely, let us consider the case when $x$ is a semi denting point of $B_X$ and $y=0$. Let $\varepsilon >0$. Then there exists a slice $S(B_{X}, x^*, \alpha)$ of $B_{X}$ with $x^*\in S_{X^*}$ and $0<\alpha<\varepsilon$ such that $S(B_{X}, x^*, \alpha) \subset B(x, \varepsilon)$. Then, for $z^*=(x^*, 0) \in S_{X^* \oplus_\infty Y^*}$, 
%consider  slice $S(B_{X \oplus_1 Y}, z^*, \alpha)$ of $B_{X \oplus_1 Y}$ such that
we have
\begin{align*}
S(B_{X \oplus_1 Y}, z^*, \alpha) & \subset S(B_{X}, x^*, \alpha) \times \varepsilon B_{Y} \\
 & \subset B(x, \varepsilon) \times \varepsilon B_{Y} \\
 & \subset B((x, 0), 2 \varepsilon).
\end{align*}
 This shows that $(x,0)$ is a semi denting point of $B_{X \oplus_1 Y}.$ 

 Similarly, we can show the other case.
\end{proof}

\begin{corollary}\label{inf1}
Let $\{X_i : i\in \I\}$ be a family of Banach spaces 
  $X=\mathop{\bigoplus_1}\limits_{i\in \I}X_i$ and $x=(x_i)_{i\in I} \in S_{X}.$ Then 
  %be such that $x_i$ is semi denting in $B_{X_i}$ for each $i$.  Then, $(x_i)$ is a semi denting point of $B_{X}$ if and only if $\|x_{i_0}\|=1$ for some $i_0 \in \mathbb{N}$.  
  $x$ is a semi denting point of $B_X$ if and only if there exists unique $i_0\in \I$ such that $x_{i_0}$ is a semi denting point of $B_{X_{i_0}}$ and $x_i=0$ for all $i\neq i_0.$
\end{corollary}

\begin{proof}
    Let  $x$ be a semi denting point. Suppose that there exist $i_1, i_2 \in \I$ such that both $x_{i_1}$ and $x_{i_2}$ are nonzero. Consider $Z=\mathop{\bigoplus_1}\limits_{i\neq i_1}X_i$ Then $(x_{i_1},z)$ is a semi denting point of $B_{X_{i_1}\bigoplus_1 Z}=B_X$ with both $x_{i_1}$ and $z$, nonzero, which is a contradiction to Theorem ~\ref{l1}. Hence, there exists a unique $i_0\in \I$ such that $x_i=0$ for all $i\neq i_0.$ Thus $(x_{i_0},0)$ is a semi denting point of $B_{X_{i_0}\bigoplus_1 Z'},$ where $Z'=(\mathop{\bigoplus_1}\limits_{i\neq i_0}X_i)$ Then from Theorem ~\ref{l1}, it directly follows that $x_{i_0}$ is a semi denting point of $B_{X_{i_0}}.$
    
%    $\|x_i\| <1$ for all $i$. Then, there exists $i_1, i_2$ such that $0 <\|x_{i_1}\|, \|x_{i_2}\|$. We may write $X=X_{i_1}\bigoplus_1 Z$, where $Z=\bigoplus_{i\neq i_1}X_i$ and $x=(x_{i_1},z)$ for some $z \neq 0$ in $Z$. We have, $\|x_{i_1}\| \neq 0$ and $\|z\| \neq 0$. By Theorem ~\ref{thmlp}, $x$ is not a semi denting point of $B_{X}$, which is a contradiction. 

    Conversely, let there exist unique $i_0\in \I$ such that $x_{i_0}$ is a semi denting point of $B_{X_{i_0}}$ and $x_i=0$ for all $i\neq i_0.$ Then from Theorem ~\ref{l1}, it directly follows that  $x=(x_{i_0},0)$ is a semi denting point of $B_{X_{i_0}\bigoplus_1 Z'}=B_X$, where $Z'=(\mathop{\bigoplus_1}\limits_{i\neq i_0}X_i)$.
    %    $\|x_{i_0}\|=1$ for some $i_0$. We may write $X^{(1)}=X_{i_0} + Z$, where $Z=\bigoplus_{i\neq i_0}X_i$ and $x=(x_{i_0},z)$. Since, $\|z\|=0$, by Theorem ~\ref{p1}, we have that $x$ is a semi denting point of $B_{X}$.
\end{proof}

Following a technique similar to that used in Theorem $\ref{l1}$ and Corollary $\ref{inf1}$, one obtains the following.

\begin{corollary}
Let $\{X_i : i\in \I\}$ be a family of Banach spaces 
  $X=\mathop{\bigoplus_1}\limits_{i\in \I}X_i$ and $x=(x_i)_{i\in I} \in S_{X}.$ Then 
  %be such that $x_i$ is semi denting in $B_{X_i}$ for each $i$.  Then, $(x_i)$ is a semi denting point of $B_{X}$ if and only if $\|x_{i_0}\|=1$ for some $i_0 \in \mathbb{N}$.  
  $x$ is a denting point of $B_X$ if and only if there exists a unique $i_0\in \I$ such that $x_{i_0}$ is a denting point of $B_{X_{i_0}}$ and $x_{i_0}=0$ for all $i\neq i_0.$
\end{corollary}

\begin{corollary}
Let $\{X_i : i\in \I\}$ be a family of Banach spaces 
  $X=\mathop{\bigoplus_1}\limits_{i\in \I}X_i$. Then the set of semi denting and denting points coincides in $X$ if and only if the set of semi denting and denting points coincides in $X_i$ for all i. 
 %   $x \in \ell_1$ is a semi denting of $B_{\ell_1}$ if and only if $x= \pm e_i$ for some $i$. Thus, in $\ell_1$, semi denting and denting points coincide.
\end{corollary}

%\begin{corollary} \label{p1}
%Let $(X_i)_i$ be a sequence of Banach spaces. If $x_k \in S_{X_k}$ is semi denting point of $B_{X_k},$ then $x_k e_k \in S_{\oplus_p X_i},$ is a semi denting point of $B_{\oplus_p X_i},$ for $1 \leq p < \infty$, where $x_k e_k=(x_k\delta_1^k, x_k \delta_2^k, \ldots)$ for $k\in \mathbb{N}$. Moreover, the converse is true for $ 1 < p < \infty$.
%\end{corollary}

\begin{lemma} \label{BDP lem inf}
\cite[Theorem 2.29 (c)]{L1} 
%Let $Z= X\oplus_{\infty} Y.$ 
For every slice $S$ 
 %(where $z^*=(x^*,y^*)\in X^*\oplus_1 Y^*$)
 of $B_{X\oplus_{\infty} Y}$, 
  %with  $x^*\neq 0$ and $y^*\neq 0$,
 %the following conditions are true.
 there exists a slice $S_1$ of $B_X$ (resp. $S_2$ of $B_Y$) and $y_0 \in B_Y$ (resp. $x_0 \in B_X$) such that 
 %\begin{center}
 $S_1 \times \{y_0\} \subset S$ $($resp. $\{x_0\} \times S_2  \subset S).$
 %\end{center}
%\begin{enumerate}
%\item There exist a slice $S(B_X,x^*,\mu_1)$ of $B_X$ and $y_0 \in B_Y$ such that\\ $S(B_X,x^*,\mu_1) \times \{y_0\} \subset S(B_Z,z^*,\alpha).$
%\item There exist a slice $S(B_Y,y^*,\mu_2)$ of $B_Y$ and $x_0 \in B_X$ such that \\ $\{x_0\} \times S(B_Y,y^*,\mu_2)  \subset S(B_Z,z^*,\alpha).$ 
%\end{enumerate} 
\end{lemma}

\begin{theorem}\label{inf sum semi dent}
Let $(x,y)\in S_{X \oplus_\infty Y}.$
%$Z=X \oplus_\infty Y$. 
Then $(x, y)$ is a semi denting point of $B_{X \oplus_\infty Y}$  if and only if $x$ is a semi denting point of $B_X$ and $y$ is a semi denting point of $B_Y$.
\end{theorem}

\begin{proof}
Let $(x, y)$ be a semi denting point of $B_{X \oplus_\infty Y}$. Also, let $\varepsilon >0$. Then, there exists a slice $S$ of $B_{X \oplus_\infty Y}$ such that 
$S \subset B((x, y), \varepsilon).$              
 Now, by Lemma $\ref{BDP lem inf}$, we can find a slice $S_1$ of $B_{X}$ and a point $y_0 \in B_{Y}$ such that 
 $S_1 \times \{y_0\} \subset S. $ Then $S_1\subset B(x, \varepsilon).$ Indeed, 
 let $\hat{x} \in S_1$. Then $(\hat{x}, y_0) \in S \subset B((x, y), \varepsilon )$. Thus, $\Vert \hat{x} - x \Vert \leqslant \max \{ \Vert \hat{x} - x \Vert, \Vert y_0 - y \Vert \} < \varepsilon.$ Hence, $\hat{x} \in B(x, \varepsilon)$. 
 %We are now only left to show that $x \in S_{X}$. Let us choose $x_{\varepsilon} \in S(B_{X}, x^*, \alpha) \cap S_{X}$, then $\Vert x_\varepsilon - x \Vert < \varepsilon. $ Now, $ \bigg \vert \Vert x_\varepsilon \Vert - \Vert x \Vert \bigg \vert < \Vert x_\varepsilon - x \Vert < \varepsilon$. This gives that $\vert 1 - \Vert x \Vert  \vert < \varepsilon$, implying further that $\Vert x \Vert > 1- \varepsilon$. Hence, $\Vert x \Vert =1$. Altogether, this concludes that 
 Consequently, $x$ is a semi denting point of $B_X$. One can similarly show the other assertion.
 
 Conversely, let $\varepsilon >0$. Let $x$ be a semi denting point of $B_X$ and $y$ be a semi denting point of $B_Y$. Then we can find slices $S(B_{X}, x^*, \alpha)$ of $B_{X}$ and $S(B_{Y}, y^*, \beta)$ of $B_{Y}$ such that 
 $ S(B_{X}, x^*, \alpha) \subset B(x, \varepsilon) ~ \text{and} ~  S(B_{Y}, y^*, \beta) \subset B(y, \varepsilon). 
 $
 Choose $0 < \delta < \min \{ \frac{\alpha}{2}, \frac{\beta}{2} \}$ and consider the slice $S(B_{X \oplus_\infty Y}, z^*, \delta)$ of $B_{X \oplus_\infty Y}$, where $z^*=(\frac{x^*}{2}, \frac{y^*}{2}).$ If $(\hat{x}, \hat{y}) \in S(B_{X \oplus_\infty Y}, z^*, \delta)$, then 
  \begin{align*}
 & \bigg (\frac{x^*}{2}, \frac{y^*}{2} \bigg ) (\hat{x}, \hat{y}) > 1 - \delta \\
 & \hspace{-0.3cm} \Rightarrow x^*(\hat{x})+ y^*(\hat{y}) > 2-2\delta \\
 & \hspace{-0.3cm} \Rightarrow x^*(\hat{x}) + 1 \geq x^*(\hat{x})+ y^*(\hat{y}) > 2-2\delta \\
 & \hspace{-0.3cm} \Rightarrow x^*(\hat{x}) > 1-2\delta > 1-\alpha .
  \end{align*}
  So, $\hat{x} \in S(B_{X}, x^*, \alpha)$. Similarly, $\hat{y} \in S(B_{Y}, y^*, \beta)$. Therefore, 
  \begin{align*}
  \Vert (\hat{x}, \hat{y}) - (x, y) \Vert_\infty & = \Vert (\hat{x} - x, \hat{y} - y) \Vert_\infty 
   = \max \{ \Vert \hat{x} - x \Vert, \Vert \hat{y} - y \Vert \} 
   < \varepsilon.
  \end{align*}
  %Hence, $(\hat{x}, \hat{y}) \in B((x,y), \varepsilon)$. 
  Thus, $S(B_{X \oplus_\infty Y}, z^*, \delta) \subset B((x, y), \varepsilon)$, which shows that $(x, y)$ is a semi denting point of $B_{X \oplus_\infty Y}$.
\end{proof}

%\begin{corollary}
 %   Let $(X_i)_i$ be a sequence of Banach spaces and $(x_i)_i \in S_{\oplus_{\infty}X_i}$ be a semi PC of $B_{\oplus_{\infty}X_i}$, then $x_i \in S_{X_i}$ is a semi PC of $B_{X_i}$, for all $i\in \mathbb{N}$.
%\end{corollary}

\begin{corollary}\label{inf sum semi dent coro}
    Let $\{X_i : i\in I\}$ be a family of Banach spaces and $X =\mathop{\bigoplus_\infty} X_i$. If $x=(x_i)_{i\in I}$ is a semi denting point of $B_{X}$, then $x_i$ is a semi denting point of $B_{X_i}$, for all $i$.
\end{corollary}
\begin{proof}
    For each $i_0\in I,$ we may write $X= X_{i_0} \oplus_{\infty} Y,$ where $Y= \mathop{\bigoplus_\infty}\limits_{i\neq i_0} X_i.$ 
    %    is $l_{\infty}$ sum of family $\{X_j : j\neq i\}.$ 
    Hence, the result follows from Theorem ~\ref{inf sum semi dent}.
\end{proof}

%\begin{remark}
%The converse of Corollary $\ref{inf sum semi dent coro}$ does not hold for infinite $I$. For example, although $\mathbb{R}$ has semi denting points, the space $l_{\infty}= \oplus_{\infty}\mathbb{R}$ has no such points as it satisfies strong diameter 2 property ( \cite[Proposition 3.5]{L2}). Similarly, when $K$ is a compact set without any isolated point, $B_{C(K)}$ does not have any semi denting point. 

    %Converse of Corollary $\ref{inf sum semi dent coro}$ is not true for infinite $I$. For example, $l_{\infty}= \oplus_{\infty}\mathbb{R}$ does not have any semi denting point (resp. semi PC, semi SCS point), though $\mathbb{R}$ has.
%\end{remark}

%%%%%%%%%%%%%%%%%%%%%%%%%%%%%%%%%
\subsection{Semi PC}
%%%%%%%%%%%%%%%%%%%%%%%%%%%%%%%%%%%%

\begin{lemma}\label{tech2}
  Let $1 <p <\infty$ and $q >0$ be such that $1/p + 1/q =1$.
  %and $0<\varepsilon<\frac{8\|x\|}{3}$.  
  Let $z=(x, y) \in S_{X \oplus_p Y}$ with $x \neq 0, y\neq 0$   be such that 
  %$x, y\neq 0,$ 
$$
\bigcap\limits_{i=-n}^n S(B_X, x_i^*,1-( x_i^*(x_0)-\del)) \subseteq B(\frac{x}{\|x\|},\frac{\e}{8}),
$$
%\cap S(B_X, -x_i^*, 1-(x_i^*(x_0)+\del)\right)
and
$$
\bigcap\limits_{i=-m}^m S(B_Y, y_i^*, 1-(y_i^*(y_0)-\del)) \subseteq B(\frac{y}{\|y\|},\frac{\e}{8}),
$$
%\cap S(B_Y, -y_i^*, 1-(y_i^*(y_0)+\del)\right)
 for some $x_{-n}^*, \ldots, x_n^*\in S_{X^*},$ $y_{-m}^*,\ldots, y_m^*\in S_{Y^*}$ and $\delta >0.$ Then $$\bigcap\limits_{i=-n}^{n} \bigcap\limits_{j=-m}^{m} S(B_{X \oplus_p Y}, z_{i,j}^*, 1-z_{i,j}^*(z_0)+\delta \|x\|^p)\bigcap (\|x\| S_X \times \|y\| S_Y) \subset B[z,\frac{\varepsilon}{2}],$$ where $z_0=(\|x\|x_0,\|y\|y_0)$ and $z_{i,j}^*=(\|x\|^{\frac{p}{q}}x_i^*,\|y\|^{\frac{p}{q}}y_j^*)$
 %, $x_i^*=sgn(i)x_{|i|}^*$, $y_j^*=sgn(j)y_{|j|}^*$ 
 for $-n\leq i\leq n, -m\leq j \leq m$.
 % $z_1^*=\|x\|^{\frac{p}{q}}x^*$, $z_2^*=\|y\|^{\frac{p}{q}}y^*$ and $z^*=(z_1^*,z_2^*)$.
\end{lemma}

\begin{proof}
    Let $w=(x_1,y_1) \in B_{X \oplus_p Y}$ be such that $\|x_1\|=\|x\|$ and $\|y_1\|=\|y\|.$
    
 Suppose that $\|z-w\|_p \geq \frac{\e}{2}$.

%Case 1: $\|x_1\|=\|x\|$ and $\|y_1\|=\|y\|$.

Let $a_1=\|x_1-x\|$, $a_2=\|y_1-y\|$, $b_1=\|x_1\|=\|x\|$, $b_2=\|y_1\|=\|y\|$.\\
Then, $a_1=\|x_1-x\| \leq \|x_1\|+\|x\|=2b_1$ and $a_2=\|y_1-y\| \leq \|y_1\|+\|y\|=2b_2$.\\
If $\|x_1-x\| \leq (\frac{\e}{8})\|x_1\|$ and $\|y_1-y\| \leq (\frac{\e}{8}) \|y_1\|$, then $\|z-w\|_p \leq \frac{\e}{8}$, which is a contradiction. Hence, $\|x_1-x\| >\frac{\e}{8}\|x_1\|$ or $\|y_1-y\| >\frac{\e}{8}\|y_1\|$. Without loss of generality, let $\|x_1-x\| >\frac{\e}{8}\|x_1\|$, i.e., $\frac{a_1}{b_1} >\frac{\e}{8}$.
Assume that $\|y_1-y\| \leq \frac{\e}{8}\|y_1\|$. i.e., $\frac{a_2}{b_2} \leq \frac{\e}{8}$.\\
%Since, $\|x_1-x\|^p + \|y_1-y\|^p \geq \e^p$, we find that $a_1^p=\|x_1-x\|^p \geq \e^p - (\frac{\e}{8})^p \geq (\frac{7\e}{8})^p$.\\
%Now, since, $\|x_1-x\|=a_1$, $\|x_1\|=b_1, \|y_1-y\|=a_2, \|x_2\|=b_2$, 
We claim that 
$x_{i'}^*(x_1) \leq (x_{i'}^*(x_0)-\del) b_1$ for some $-n \leq i' \leq n$. 
Indeed, if $x_i^*(x_1) >(x_i^*(x_0)-\del) b_1$ for all $i$, then $x_i^*(\frac{x_1}{\|x_1\|}) >x_i^*(x_0)-\del $, which further implies that $\|\frac{x_1}{\|x_1\|}-\frac{x}{\|x\|}\| <\frac{\e}{8} <\frac{a_1}{b_1}$. i.e., $\|x_1-x\| <\|x_1-x\|$, which leads to a contradiction. Hence, our claim follows.

%(Reason: If $x_i^*(x_1) >(x_i^*(x_0)-\del) b_1$ for all $i$, then $x_i^*(\frac{x_1}{\|x_1\|}) >x_i^*(x_0)-\del $, which implies $\|\frac{x_1}{\|x_1\|}-\frac{x}{\|x\|}\| <\frac{\e}{8} <\frac{a_1}{b_1}$, that is, $\|x_1-x\| <\|x_1-x\|$, a contradiction).
\noindent
Let us choose any $-m \leq j' \leq m$ such that $y_{j'}^*(y_1) \leq y_{j'}^*(\|y\|y_0)$ 
%(exists because either $y_{|j'|}^*(y_1) \leq y_{|j'|}^*(\|y\|y_0)$ or $-y_{|j'|}^*(y_1) \leq -y_{|j'|}^*(\|y\|y_0)$.
 Then, we have
\beqa
z_{i',j'}^*(w) & = &\|x\|^{\frac{p}{q}}x_{i'}^*(x_1) + \|y\|^{\frac{p}{q}}y_{j'}^*(y_1)\\
& \leq & \|x\|^{\frac{p}{q}}(x_{i'}^* (b_1 x_0)-\del b_1) + \|y\|^{\frac{p}{q}}y_{j'}^*(\|y\|y_0)\\
&=&\|x\|^{\frac{p}{q}}(x_{i'}^* (\|x\| x_0)-\del b_1) + \|y\|^{\frac{p}{q}}y_{j'}^*(\|y\|y_0)\\
&=&(\|x\|^{\frac{p}{q}}x_{i'}^*) (\|x\| x_0)-\del \|x\|^{\frac{p}{q}} b_1 + (\|y\|^{\frac{p}{q}}y_{j'})^*(\|y\|y_0)\\
& = & z_{i',j'}^*(z_0)-\del b_1^p.
\eeqa
So, $z^*(w) > z_{i,j}^*(z_0)-\del b_1^p$ for all $i,j$ implies that $\|w-z\|_p <\frac{\e}{2}$.
%In the case when $\|y_1-y\| > (\frac{\e}{4})\|y_1\|$, i.e., $\frac{a_2}{b_2} > \frac{\e}{4}$, same argument can be repeated.
\end{proof}

\begin{theorem} \label{thmpc} 
Let $1<p<\infty.$ Then $(x, y) $ is a semi PC of $B_{X \oplus_p Y}$ if and only if $\frac{x}{\Vert x \Vert}  $ is a semi PC of $B_X$ whenever $x \neq 0$ and $\frac{y}{\Vert y \Vert} \ $ is a semi PC  of $B_Y$ whenever $y \neq 0$.
\end{theorem}

\begin{proof}
Let $q >0$ be such that $1/p + 1/q = 1$. Let $(x,y)$ be a semi PC of $B_{X \oplus_p Y}$. We will prove that $\frac{x}{\Vert x \Vert}  $ is a semi PC of $B_X$ whenever $x \neq 0$. The other case follows in a similar way. So, let $x\neq 0.$ Let $ 0 < \varepsilon < \Vert x \Vert$. Then there exists a weakly open set $V$ of $B_{X \oplus_p Y}$ such that $V \subseteq B((x, y), \varepsilon).$ Let us choose $\hat{z}=(\hat{x}, \hat{y}) \in V \cap S_{X \oplus_p Y}$. Now, we can find a basic weakly open subset 
$$
\tilde{V} = \{ z \in B_{X \oplus_p Y}: \vert z_i^*(z-\hat{z}) \vert < 1, \ 1 \leq i \leq n \},
$$ 
of $B_{X \oplus_p Y}$ such that $\tilde{V} \subset V$, where $z_i^* = (x_i^*, y_i^*) \in X^* \oplus_q Y^*$ where $q$ satisfies $1/p + 1/q = 1$. We claim that $\hat{x}\neq 0.$ 
Suppose, to the contrary, that $\hat{x}= 0$ so that $\hat{y} \in S_{Y}.$ Let us choose $y_0 \in \{ y \in B_{Y} : \vert y_i^*(y-\hat{y})\vert <1, \ i=1,2,...n\}.$
Then $(0, y_0) \in \tilde{V} \subset V \subset B((x, y), \varepsilon)$.
So, $$
\Vert x \Vert^p \leq \Vert x \Vert^p + \Vert y_0 - y \Vert^p = \|(0, y_0)- (x,y)\|_p^p< \varepsilon^p,
$$ 
which leads to a contradiction. Hence, our claim follows. 

Let us now consider the weakly open set $$V_1=\{x_0 \in B_{X} : \big| x_i^* \bigg (x_0-\frac{\hat{x}}{\Vert \hat{x} \Vert} \bigg ) \big| < \frac{1}{2\Vert \hat{x} \Vert}, \ i=1,2,..n\},$$ 
of $B_X$. 
Let $x_0\in V_1.$ Choose $y_0\in B_Y$ such that $\big| y_i^* \bigg (\Vert \hat{y} \Vert y_0-\hat{y} \bigg ) \big| < \frac{1}{2}$ for all $i=1,2,\ldots,n.$ Then we can obtain that 
$(\Vert \hat{x} \Vert x_0, \Vert \hat{y} \Vert y_0)\in \tilde{V} \subset V \subset B((x, y), \varepsilon).$ Thus, $\Vert \Vert \hat{x} \Vert x_0 - x \Vert < \varepsilon$.
Consequently,
\begin{align*}
 \Vert \Vert x \Vert x_0 - x \Vert & \leq \Vert \Vert x \Vert x_0 - \Vert \hat{x} \Vert x_0 \Vert + \Vert \Vert \hat{x} \Vert x_0 - x \Vert \\
& < \Vert x_0 \Vert \vert \Vert x \Vert - \Vert \hat{x} \Vert \vert + \varepsilon \\
& < \varepsilon + \varepsilon = 2 \varepsilon.
\end{align*}
Finally, 
$$
\bigg \Vert x_0 - \frac{x}{\Vert x \Vert} \bigg \Vert = \frac{\Vert \Vert x \Vert x_0 - x \Vert}{\Vert x  \Vert} < \frac{2 \varepsilon}{\Vert x \Vert}.
$$
Hence, $V_1 \subset B \bigg (\frac{x}{\Vert x \Vert}, \frac{2 \varepsilon}{\Vert x \Vert} \bigg ).$

Conversely, suppose that $\frac{x}{\Vert x \Vert}  $ is a semi PC of $B_X$ whenever $x \neq 0$ and $\frac{y}{\Vert y \Vert} \ $ is a semi PC  of $B_Y$ whenever $y \neq 0$.

Case 1: Either $x=0$ or $y=0$.

Without loss of generality, let $y=0.$ Then $x\neq 0$ and hence a semi PC of $B_X.$ Let $\varepsilon >0.$
%$y=0$ and $x$ is a semi PC of $B_X$. 
Thus, there exists a weakly open set $U$ of $B_{X}$ such that $U \subset B(x, \varepsilon)$. 
%Now, $U$ must contain a nonempty finite intersection of slices, say $S(B_{X}, x_{i}^*, \alpha_i), ~ 1 \leq i \leq n$ 
Then, we can find slices $S(B_{X}, x_{i}^*, \alpha_i), ~ 1 \leq i \leq n$ of $B_X$ such that $\bigcap\limits_{i=1}^n S(B_{X}, x_{i}^*, \alpha_i) \subset U$. 
Choose $x_0\in S(B_X, x_i^*, \alpha_i).$ Let $x_0^*\in S_{X^*}$ be such that $x_0^*(x_0)=1$, and consider $\bigcap\limits_{i=0}^{n} S(B_X, x_i^*, \alpha_i),$ where $\alpha_0>0$ with $(1-(1-\alpha_0)^p)^{1/p}< \varepsilon.$ Define $z_i^*=(x_i^*,0)$ for each $i=0,\ldots,n.$ Then $$(x_0,0)\in \bigcap\limits_{i=0}^{n}S(B_{X\oplus_p Y}, z_i^*, \alpha_i)\subset \bigcap\limits_{i=0}^{n} S(B_X, x_i^*, \alpha_i) \times \varepsilon B_Y\subset B((x,0), 2\varepsilon).$$

%By \cite[Lemma 2.3]{BS1}, we get a slice $S(B_{X \oplus_p Y}, z_i^*, \mu_i)$ of $B_{X \oplus_p Y}$ such that  $$  S(B_{X \oplus_p Y}, z_i^*, \mu_i) \subset S(B_{X}, x_{i}^*, \alpha_i) \times \varepsilon B_{Y}.$$ Let us now chose $0 < \mu < \min \{ \mu_1, \cdots, \mu_n \}$. Then,  \begin{align*} \cap_{i=1}^n S(B_{X \oplus_p Y}, z_i^*, \mu) & \subset \cap_{i=1}^n [ S(B_{X}, x_{i}^*, \alpha_i) \times \varepsilon B_{Y}] \\& \subset U \times \varepsilon B_{Y} \\ & \subset B(x, \varepsilon) \times \varepsilon B_{Y} \\& \subset B((x, 0), \varepsilon). \end{align*}Thus, $\cap_{i=1}^n S(B_{X \oplus_p Y}, z_i^*, \mu) \subset B((x, 0), \varepsilon),$ which shows that $(x, 0)$ is a semi PC point $B_{X \oplus_p Y}$.

%Similarly, the case when $x=0$ can be proved analogously.

Case 2: Both $x \neq 0$ and $y \neq 0$.
%Both $x$ and $y$ are nonzero.
 
Then $\frac{x}{\|x\|} $ is a semi PC of $B_X$ and $\frac{y}{\|y\|} $ is a semi PC of $B_Y$. Let $\e >0$. 
Choose $\del >0$, $x_0 \in S_X$, $y_0 \in S_Y$, $x_1^*,\ldots, x_n^*\in S_{X^*}$ and $y_1^*,\ldots, y_m^*\in S_{Y^*}$ 
%$\{x_i^*\}_{i=1}^n \subseteq S_{X^*}$ and $\{y_i^*\}_{i=1}^m \subseteq S_{Y^*}$ 
such that
\[
\bigcap_{i=1}^n \left(S(B_X, x_i^*,1-( x_i^*(x_0)-\del)) \cap S(B_X, -x_i^*, 1-(x_i^*(x_0)+\del)\right)\subseteq B(\frac{x}{\|x\|},\frac{\e}{8}),
\]
and
\[
\bigcap_{i=1}^m \left(S(B_Y, y_i^*, 1-(y_i^*(y_0)-\del)) \cap S(B_Y, -y_i^*, 1-(y_i^*(y_0)+\del)\right)\subseteq B(\frac{y}{\|y\|},\frac{\e}{8}).
\]
We may assume that there exists $1 \leq i_0 \leq n$ and $1 \leq j_0 \leq m$ such that $x_{i_0}^*(x_0)=1$ and $y_{j_0}^*(y_0)=1$.\\
For $-n \leq i \leq n$ and $-m \leq j \leq m$, let $x_i^*=sgn(i)x_{|i|}^*$ and $y_j^*=sgn(j)y_{|j|}^*$.\\
Let $z_{i,j}^*=(\|x\|^{\frac{p}{q}}x_i^*,\|y\|^{\frac{p}{q}}y_j^*)$ and $z_0=(\|x\|x_0,\|y\|y_0)$. Clearly, $z_{i,j}^* \in S_{Z^*}$, $z_0 \in B_Z$ and so $z_{i_0,j_0}^*(z_0)=1$. Also let $z=(x,y)$. 
Then by Lemma ~\ref{tech2}, $$\bigcap\limits_{i=-n}^{n} \bigcap\limits_{j=-m}^{m} S(B_{X \oplus_p Y}, z_{i,j}^*, 1-z_{i,j}^*(z_0)+\delta \|x\|^p)\bigcap (\|x\| S_X \times \|y\| S_Y) \subset B[z,\frac{\varepsilon}{2}].$$

Since $\ell_p^2$ is uniformly convex, let $\Delta_p(\e)$ be the modulus of convexity. Choose $\alpha$ such that $0 <\alpha <\frac{\e}{2}$ and $\Delta_p(\frac{\alpha}{2})+\alpha <\del \|x\|^p$. \\    
Let $w=(x_1,x_2) \in B_{X\oplus_p Y}$ be such that $z_{i,j}^*(w) > z_{i,j}^*(z_0) - \Del_p(\frac{\alpha}{2})$, for all $-n \leq i \leq n$ and $-m \leq j \leq m$. Then,
\beqa
 \|x\|^{\frac{p}{q}} \|x_1\| + \|y\|^{\frac{p}{q}}\|x_2\| & > &\|x\|^{\frac{p}{q}}x_{i_0,j_0}^*(x_1)+\|y\|^{\frac{p}{q}}y_{i_0,j_0}^*(y_1)\\
 & > & z_{i_0,j_0}^*(z_0)-\Del_p(\frac{\alpha}{2})\\
 &=& 1 - \Del_p(\frac{\alpha}{2}).
\eeqa
Let $\tilde{x} := (\|x\|)$ and $\tilde{y}:=(\|y\|)$. We have, $\tilde{z}=(\tilde{x}, \tilde{y}) \in S_{\ell_p^2}$, $\tilde{w} := (\|x_1\|,\|x_2\|) \in B_{\ell_p^2}$ and $u := (\|x\|^{\frac{p}{q}},\|y\|^{\frac{p}{q}}) \in S_{\ell_q^2}$. Then,
\[
\tilde{z}(u) = \|x\| \|x\|^{\frac{p}{q}} + \|y\| \|y\|^{\frac{p}{q}}= 1 \quad \mbox{and} \quad \tilde{w}(u) > 1 - \Del_p(\frac{\alpha}{2}).
\]
%Here, we may consider $u$ acting as a functional on $\ell_p^2$. 
Also, $\tilde{w}(u) > 1 - \Del_p(\frac{\alpha}{2})$, which implies that $\frac{\tilde{z}(u)+\tilde{w}(u)}{2} >1-\Del_p(\frac{\alpha}{2})$. i.e., $\frac{\|\tilde{z}+\tilde{w}\|}{2} >1-\Del_p(\frac{\alpha}{2})$. Now, by the definition of the modulus of convexity, we have
\[
\|\tilde{z}-\tilde{w}\|_p = \left( \bigg|\|x_1\| - \|x\|\bigg|^p + \bigg|\|x_2\| - \|y\|\bigg|^p\right)^{\frac{1}{p}}< \alpha.
\]
Let $c=(c_1,c_2)$, where $c_1 = \frac{x_1\|x\|}{\|x_1\|}$ and $c_2=\frac{x_2\|x\|}{\|x_2\|}$. Clearly, $\|c_1\|=\|x\|$ and $\|c_2\|=\|y\|$. Hence,
\[
\|w-c\|_p= \left(\bigg|\|x_1\| - \|x\|\bigg|^p +  \bigg|\|x_2\| - \|y\|\bigg|^p\right)^{\frac{1}{p}}< \alpha.
\]
Since for any $i,j$, \ $z_{i,j}^*(w) >z_{i,j}^*(z_0)-\Del_p(\frac{\alpha}{2})$, we have $z_{i,j}^*(c) \geq z_{i,j}^*(w) -\|w-c\|_p > z_{i,j}^*(z_0)-\Del_p(\frac{\alpha}{2}) - \alpha >z_{i,j}^*(z_0) - \del \|x\|^p$.
%by the choice of $\alpha$. By the first case,
Thus,
\[
\|z-c\|_p <\frac{\e}{2},
\]
and finally,
\[
\|z-w\|_p \leq \|z-c\|_p + \|c-w\|_p <\e.
\]
\end{proof}

\begin{theorem} \label{thmPClp}
Let $\{X_i : i\in I\}$ be a family of Banach spaces and
$X = \bigoplus_{p}X_{i}, ~ 1 < p < \infty .$
Also let $x=(x_i)_{i=1}^\infty \in S_{X}$. Then $x$ is a semi PC of $B_{X}$ if and only if $\frac{x_i}{\|x_i\|}$ is a semi PC of $B_{X_i}$ whenever $x_i \neq 0$.  
\end{theorem}

\begin{proof}
Let $x$ be semi PC of $B_X$ and $i_0 \in \N$ be such that $x_{i_0} \neq 0$. Let $\e >0$ be given.
%Let $(x_i)_{i=1}^\infty \in S_{X_p}$ is a semi PC of $B_{X_p}$. 
We may identify $X$ as $X_{i_0}\bigoplus_p Y$ where $Y=\mathop{\bigoplus_p}\limits_{i\neq i_0} X_i$. So, by Theorem $\ref{thmpc}$, we have $\frac{x_{i_0}}{\|x_{i_0}\|}$ is a semi PC in $B_{X_{i_0}}$.

Conversely, suppose that $\frac{x_i}{\|x_i\|}$ is a semi PC of $B_{X_i}$ for all $i$ such that $x_i \neq 0$ and $\e >0.$  
%Since $(x_i)_{i=1}^\infty \in S(X_p)$, $\sum_{i=1}^\infty \|x_i\|^p = 1$. 
Choose $n_0 \in \N$ such that $\sum\limits_{i >n_0}\|x_i\|^p < (\frac{\e}{8})^p$. So, $\sum\limits_{i \leq n_0}\|x_i\|^p >1-(\frac{\e}{8})^p >(1-\frac{\e}{8})^p$.

Let $W:=\bigoplus_{i=1}^{n_0} X_i$. Let $z=(x_i)_{i=1}^\infty$, $w=(x_i)_{i=1}^{n_0}$ and $\tilde{w}=\frac{w}{\|w\|_p}$. By extending Theorem ~\ref{thmpc} to finitely many case, there exist $0< \del <\frac{\e}{8}$ and $\{w_i^*\}_{1 \leq i \leq k} \subseteq S_{W^*}$ and $w_0\in B(W)$ such that
\[
\bigcap_{1 \leq i \leq k} S(B_W, w_i^*, 1-w_i^*(w_0)-\del) \subseteq B(\tilde{w},\frac{\e}{8}).
\]
Just as in Theorem ~\ref{thmpc}, we may assume that for some $1 \leq j_0 \leq k$, $w_{j_0}^*(w_0)=1$.\\
For $1 \leq j \leq k$, let $z_j^*$ be defined as 
\[
z_j^*(i) =
\begin{cases}
 w_j^{*}(i) &\text{if } i \leq n_0, \\
 0 &\text{otherwise}.
 \end{cases}
 \]
 Clearly $z_j^* \in S(X^*)$ for all $1 \leq j \leq k$. 
 Also, let \[
z_0(i) =
\begin{cases}
 w_0(i) &\text{if } i \leq n_0, \\
 0 &\text{otherwise}.
 \end{cases}
 \]
 Clearly $z_0 \in B_W$.\\
 Let $y:=(y_i)_{i=1}^\infty \in B_{X}$ and $v:=(y_i)_{i=1}^{n_0}$. Let $z_j^*(y) >(z_j^*(z_0)-\del^p)^{\frac{1}{p}}$ for all $1 \leq j \leq k$. So, $w_j^*(v) >(w_j^*(w_0)-\del^p)^{\frac{1}{p}}>w_j^*(w_0)-\del$ for all $1 \leq j \leq k$. So, $\|v-\tilde{w}\|_p <\frac{\e}{8}$.\\
  Also, $w_{j_0}^*(v) >(w_{j_0}^*(w_0)-\del^p)^{\frac{1}{p}}$, that is, $w_{j_0}^*(v) >(1-\del^p)^{\frac{1}{p}}$. So, $\|v\|_p^p >1-\del^p$, that is,  $\sum_{i \leq n_0}\|y_i\|^p >1-\del^p$. Therefore, $\sum_{i >n_0}\|y_i\|^p <\del^p$ which implies $\left(\sum_{i>n_0} \|y_i\|^p\right)^{\frac{1}{p}} <\del <\frac{\e}{8}$.\\
 So, 
 \begin{align*}
     \|w-v\|_p & \leq \|w-\tilde{w}\|_p + \|\tilde{w}-v\|_p \\
     & \leq \frac{\e}{8} + \frac{\e}{8} =\frac{\e}{4} .
 \end{align*}
 Finally, we have
\beqa
 \|y-z\|_p &=& \left(\sum_{i=1}^\infty\|y_i-x_i\|^p\right)^{\frac{1}{p}}\\
 &\leq &\left(\sum_{i \leq n_0}\|y_i-x_i\|^p\right)^{\frac{1}{p}} + \left(\sum_{i >n_0} \|y_i-x_i\|^p\right)^{\frac{1}{p}}\\
 &\leq& \|v-w\|_p +  \left(\sum_{i >n_0} \|y_i\|^p\right)^{\frac{1}{p}} + \left(\sum_{i >n_0} \|x_i\|^p\right)^{\frac{1}{p}}\\
 &\leq& \frac{\e}{4} + \frac{\e}{8} + \frac{\e}{8}\\
 &<& \e.
\eeqa
\end{proof}

%proceeding with similar technique as in Theorem \ref{thmpc} and \ref{thmPClp}, we have following.
Using a similar technique to that employed in Theorem \ref{thmpc} and Theorem \ref{thmPClp}, we obtain the following.
\begin{corollary}
 Let $\{X_i : i\in I\}$ be a family of Banach spaces and
$X = \mathop{\bigoplus_{1}}X_{i}$ and $x=(x_i)_{i=1}^\infty \in S_{X}$. If $x$ is a semi PC of $B_{X}$, then $\frac{x_i}{\|x_i\|}$ is a semi PC of $B_{X_i}$ whenever $x_i \neq 0$.   
\end{corollary}

\begin{lemma}\label{BHP lem inf} 
\cite{ALN1}
     Let $W$ be a nonempty weakly open subset of $B_{X \oplus_\infty Y}$. Then there exist weakly open subsets $U$ and $V$ of $B_X$ and $B_Y$, respectively, such that $U \times V \subset W.$
\end{lemma}

\begin{theorem}\label{inf sum semi pc}
%Let $X$ and $Y$ be Banach spaces. Then 
Let $(x, y) \in  S_{X \oplus_\infty Y}.$ Then $(x,y)$ is a semi PC of $B_{X \oplus_\infty Y}$ if and only if $x$ is a semi PC of $B_X$ and $y $ is a semi PC of $B_Y$.
\end{theorem}

\begin{proof}
Let $\varepsilon >0$ and $(x, y)$ be a semi PC of $B_{X \oplus_\infty Y}$. Then there exists a weakly open set $U$ of $B_{X \oplus_\infty Y}$ such that $U \subset B((x, y), \varepsilon)$. By Lemma $\ref{BHP lem inf}$, we can find weakly open subsets $U_1\subset B_X$ and $U_2\subset B_Y$ such that $U_1 \times U_2 \subset U$. Let us now consider $\hat{x} \in U_1$ and $\hat{y} \in U_2$. Then, $(\hat{x}, \hat{y}) \in U \subset B((x, y), \varepsilon)$. This implies that 
\begin{align*}
\max \{ \Vert \hat{x} -x \Vert , \Vert \hat{y} - y \Vert \}=   \Vert (\hat{x}, \hat{y}) - (x, y) \Vert < \varepsilon.
\end{align*}   
Thus $U_1 \subset B(x, \varepsilon)$ and $U_2 \subset B(y, \varepsilon)$, proving that $x $ is a semi PC of $B_X$ and $y $ is a semi PC of $B_Y$.

Conversely, let $x $ be a semi PC of $B_X$ and $y $ be a semi PC of $B_Y$. Also let $\varepsilon >0$. Then, there exists nonempty weakly open sets, say $U_1\subset B_X$ and $U_2\subset B_Y$ such that $U_1 \subset B(x, \varepsilon)$ and $U_2 \subset B(y, \varepsilon)$. 
Then proceeding similarly as in \cite[Proposition 2.11]{BS}, we get a weakly open set $ U \subset B_{X \oplus_\infty Y}$ such that $U \subset U_1 \oplus_\infty U_2.$
 %Since slices form a subbase for weakly open subsets of $B_X$ and $B_Y$, we can find slices $S(B_X, x_i^*, \alpha_i)$ of $B_X$ and $S(B_Y, y_i^*, \beta_i)$ of $B_Y$ such that $$\cap_{i=1}^n S(B_{X}, x_i^*, \alpha_i) \subset U_1, ~ \text{and} ~ \cap_{j=1}^m S(B_{Y}, y_j^*, \beta_j) \subset U_2.$$ Without loss of generality, let $n \geq m$. Then we get slices of $B_{X \oplus_\infty Y}$, say $S(B_{X \oplus_\infty Y}, z_i^*, \delta_i)$, where $z_i^*=(\frac{x_i^*}{2}, \frac{y_i^*}{2})$ and     $$ 0 < 2 \delta_i < \left\{ \begin{array}{rcl}		\min \{\alpha_i, \beta_i \} 		&\mbox{for} & i=1, \cdots, m \\	  \min \{\alpha_i, \beta_m \}		&\mbox{for} & i=m+1, \cdots, n,	\end{array} \right.	$$ satisfying $$ S(B_{X \oplus_\infty Y}, z_i^*, \delta_i) \subset S(B_{X}, x_i^*, \alpha_i) \oplus_\infty S(B_{Y}, y_i^*, \beta_i), ~ \forall i=1, \cdots, m, $$ and $$  S(B_{X \oplus_\infty Y}, z_i^*, \delta_i) \subset S(B_{X}, x_i^*, \alpha_i) \oplus_\infty S(B_{Y}, y_i^*, \beta_m), ~ \forall i=m+1, \cdots, n. $$ Then, we have  $$ S(B_{X \oplus_\infty Y}, z_i^*, \delta_i) \subset S(B_{X}, x_i^*, \alpha_i) \oplus_\infty S(B_{Y}, y_i^*, \beta_i). $$ This further implies that $$ \cap_{i=1}^n S(B_{X \oplus_\infty Y}, z_i^*, \delta_i) \subset \cap_{i=1}^n S(B_{X}, x_i^*, \alpha_i) \oplus_\infty \cap_{i=1}^m S(B_{Y}, y_i^*, \beta_i). $$
Therefore,
  \begin{align*}
U \subset U_1 \oplus_\infty U_2 \subset B(x, \varepsilon) \oplus_\infty B(y, \varepsilon)  \subset B((x, y), \varepsilon).
 \end{align*}
 \enlargethispage{\baselineskip}
 This proves that $(x, y) $ is a semi PC of $B_{X \oplus_\infty Y}$.
\end{proof}

%Proceeding similarly as  Corollary ~\ref{inf sum semi dent coro}, we have 
\begin{corollary}\label{inf sum PC coro}
    Let $\{X_i : i\in I\}$ be a family of Banach spaces and $X=\mathop{\bigoplus_\infty} X_i$. If $x=(x_i)_{i\in I}$ is a semi PC in $B_{X}$, then $x_i$ is a semi PC in $B_{X_i}$ for all $i\in I$.
\end{corollary}
\begin{proof}
    For each $i_0\in I,$ we may write $X= X_{i_0} \oplus_{\infty} Y,$ where $Y= \mathop{\bigoplus_\infty}\limits_{i\neq i_0} X_i.$ 
    % is $l_{\infty}$ sum of family $\{X_j : j\neq i\}.$ 
    Hence the result follows from Theorem ~\ref{inf sum semi pc}.  
\end{proof}

%\begin{remark}
%The converse of Corollary $\ref{inf sum PC coro}$ does not hold for infinite $I$.
 %   Since  $l_{\infty}= \oplus_{\infty}\mathbb{R}$ satisfies strong diameter 2 property, so the unit ball cannot have small weakly open sets, hence cannot have any point which is semi PC.
%\end{remark}

%The proof of the following theorem follows from Theorem ~\ref{thmPClp}. Hence, we skip it.{ \bf THIS IS UNPROFFESIONAL AND  LOUSY .PLEASE CHANGE}
%\begin{theorem} $x$ is a semi PC of $B_X$ if and only if $(x, 0)$ is a semi PC $B_{X\oplus_1 Y}$.
%\end{theorem}{ \bf PLEASE CHANGE}

%\begin{question}
%Let $X$ and $Y$ be Banach spaces and $(0,0) \neq (x,y) \in S_{X \oplus_p Y}, x \neq 0 ,y \neq 0, ~ 1<p<\infty$. Then, under what condition is $(x,y) \in $ semi dent $(B_{X \oplus_p Y})$ (resp. $(x,y) \in$ semi PC $(B_{X \oplus_p Y})$)?   
%\end{question}
%%%%%%%%%%%%%%%%%%%%%%%%%%%%%%%%%%%%%%%%%%%%%%
\subsection{\text{Semi} SCS points}
%%%%%%%%%%%%%%%%%%%%%%%%%%%%%%%%%%%%%%%%%%%%%%%%%%%%

\begin{theorem} \label{scsT}
Let $(x,y) \in S_{X \oplus_1 Y}$. Then $(x, y)$ is a semi SCS  point of $B_{X \oplus_1 Y}$ if and only if  $\frac{x}{\|x\|}$ is a semi SCS point of $B_X$  whenever $x\neq 0$ and  $\frac{y}{\|y\|}$ is a semi SCS point of  $B_Y$ whenever $y\neq 0$.
\end{theorem}

\begin{proof}
Let $(x,y)$ be a semi SCS point of $B_{X \oplus_1 Y}$. We will prove that $\frac{x}{\|x\|}$ is a semi SCS point of $B_X$ whenever $x\neq 0$. The other case follows similarly. So, let $x\neq 0$ and $0<\varepsilon<\|x\|.$   Then, there exists a convex combination of slices, say $S=\sum\limits_{i=1}^{n} \lambda_i S(B_{X \oplus_1 Y}, (x_i^*, y_i^*), \alpha_i)$ of $B_{X \oplus_1 Y}$ such that $S \subset B((x,y), \varepsilon),$ where $(x_i^*, y_i^*)\in S_{X^*\oplus_{\infty} Y^*}$ $\forall i= 1,\ldots, n$. Consider 
$$
Q=\{i\in \{1,\ldots, n\} : \|x_i^*\|=1\} ~ \text{and} ~ Q^c=\{1,\ldots, n\} \setminus Q.
$$
We now claim that $Q\neq \emptyset$ and $Q^c\neq \emptyset$. Assuming on the contrary, let either of $Q$ and $Q^c$ be the empty set. Without loss of generality, let $Q=\emptyset$ and $\hat{z_i} \in S(B_Y,  y_i^*, \alpha_i)$ for all $i=1,\ldots,n$. Then $\sum\limits_{i=1}^{n} \lambda_i (0,\hat{z_i}) \in S$ and so,
$$
\varepsilon > \|(x, y) - \sum\limits_{i=1}^{n} \lambda_i (0,\hat{z_i})\|_1 \geqslant \|x\|,
$$ 
which leads to a contradiction, proving our claim.\\
Let $\lambda_Q= \sum\limits_{i\in Q} \lambda_i$ and $\lambda_{Q^c}= \sum\limits_{j\in Q^c} \lambda_j.$ 
Consider $S_1= \sum\limits_{i\in Q} \frac{\lambda_i}{\lambda_Q} S(B_{X}, x_i^*, \alpha_i)$ and $S_2=\sum\limits_{j\in Q^c} \frac{\lambda_j}{\lambda_{Q^c}} S(B_{Y}, y_j^*, \alpha_j)$. Let $\sum\limits_{i\in Q} \frac{\lambda_i}{\lambda_Q} a_i \in S_1$. Define $$z_i=    \left\{ \begin{array}{rcl}
		(a_i,0), &\mbox{if}
		& i\in Q , \\
		(0,\hat{z_i}), &\mbox{if}
		& i \in Q^c .
	\end{array} \right.$$
	where $a_i \in S(B_X, x_i^*, \alpha_i) $ and $\hat{z_i} \in S(B_{Y}, y_i^*, \alpha_i)$. Then $\sum\limits_{i=1}^{n} \lambda_i z_i \in S$. Also 
    $$
    \varepsilon > \|(x, y) - \sum\limits_{i=1}^{n} \lambda_i z_i\|_1= \|x- \sum\limits_{i\in Q} \lambda_i a_i\|+ \|y- \sum\limits_{j\in Q^c} \lambda_j \hat{z_j} \|.
    $$
	This gives us that
    $$
    \lambda_Q= \sum\limits_{i\in Q} \lambda_i \geqslant \|\sum\limits_{i\in Q} \lambda_i a_i\| > \|x\|-\varepsilon ,
    $$
    and 
    $$
    \lambda_Q= 1- \lambda_{Q^c}= 1- \sum\limits_{j\in Q^c} \lambda_j \leqslant 1- \|\sum\limits_{j\in Q^c} \lambda_j \hat{z_j}\| < 1-\|y\|+\varepsilon = \|x\|+\varepsilon.
    $$
	Hence $|\|x\|- \lambda_Q| < \varepsilon$. Finally, 
	\begin{align*}
\bigg \|\frac{x}{\|x\|} - \sum\limits_{i\in Q} \frac{\lambda_i}{\lambda_Q} a_i \bigg \| & \leqslant \bigg \|\frac{x}{\|x\|} - \sum\limits_{i\in Q} \frac{\lambda_i}{\|x\|} a_i \bigg \| + \bigg \| \sum\limits_{i\in Q} \frac{\lambda_i}{\|x\|} a_i - \sum\limits_{i\in Q} \frac{\lambda_i}{\lambda_Q} a_i \bigg \|\\
& < \frac{\varepsilon}{\|x\|}+ \bigg |\frac{1}{\|x\|}- \frac{1}{\lambda_Q} \bigg |\\
& < \frac{\varepsilon}{\|x\|}+ \frac{\varepsilon}{\|x\| \lambda_Q}\\
& < \frac{\varepsilon}{\|x\|}+ \frac{\varepsilon}{\|x\| (\|x\|-\varepsilon)} = \varepsilon^{'} ~ \text{(say)}.
\end{align*}
Hence $S_1 \subset B \bigg (\frac{x}{\|x\|},  \frac{\varepsilon}{\|x\|}+ \varepsilon^{'} \bigg )$. Thus $\frac{x}{\|x\|} \in $  semi SCS $(B_{X})$.

Conversely, suppose that $\frac{x}{\|x\|}$ is a semi SCS point of $B_X$  whenever $x\neq 0$ and  $\frac{y}{\|y\|}$ is a semi SCS point of  $B_Y$ whenever $y\neq 0$.

Case 1: Either $x=0$ or $y=0.$

Without loss of generality, let $y=0$. Then $x\in S_X$ is a semi SCS point of $B_X.$
Then there exists a convex combination of slices $S=\sum_{i=1}^n \lambda_i S(B_{X}, x_i^*, \alpha_i)$ of $B_{X}$ such that $S \subset B(x, \varepsilon)$. By \cite[Lemma 2.1]{BS}, we can find a slice $S(B_{X \oplus_1 Y}, z_i^*, \beta_i)$ of $B_{X \oplus_1 Y}$ such that 
$$
S(B_{X \oplus_1 Y}, z_i^*, \beta_i) \subset S(B_{X}, x_i^*, \alpha_i) \times \varepsilon B_{Y}.
$$
Thus, 
\begin{align*}
\sum_{i=1}^n \lambda_i S(B_{X \oplus_p Y}, z_i^*, \beta_i) & \subset \sum_{i=1}^n \lambda_i (S(B_{X}, x_i^*, \alpha_i) \times \varepsilon B_{Y}) \\
& \subset B(x, \varepsilon) \times \varepsilon B_{Y} \\
& \subset B((x, 0), \varepsilon).
\end{align*}
This proves that $(x, 0)$ is a semi SCS point of $B_{X \oplus_p Y}$.

Case 2: Both $x$ and $y$ are nonzero.

Then $\frac{x}{\|x\|}$ be a semi SCS point of $B_X$ and $\frac{y}{\|y\|}$ be a semi SCS point of $B_Y$, respectively. Then, by previous arguments, we find that $(\frac{x}{\|x\|}, 0)$ and $(0, \frac{y}{\|y\|})$ are semi SCS points of $B_{X \oplus_1 Y}.$ We can now observe that $(x, y) = \Vert x \Vert (\frac{x}{\Vert x \Vert}, 0)+\Vert y \Vert(0, \frac{y}{\Vert y \Vert})$ with $\Vert x \Vert + \Vert y \Vert=1.$ Since the set of semi SCS points is convex, we conclude that $(x, y) $ is a semi SCS point of $B_{X \oplus_1 Y}$.

\end{proof}

%\begin{question}
%Does the converse of Corollary \ref{scslp} hold? 
%\end{question}

\begin{theorem}\label{inf sum semi scs}
%Let $X$ and $Y$ be Banach spaces. 
If $(x, y)$ is a semi SCS point of $B_{X \oplus_\infty Y}$, then $x $ is a semi SCS  point of $B_X$ and $y $ is a semi SCS  point of $B_Y$.
\end{theorem}     
\begin{proof}        
Let $\varepsilon >0$ and $(x, y) $ be a semi SCS point of $B_{X \oplus_\infty Y}$. Then there exists a collection $\{S_i:=S(B_{X \oplus_\infty Y},z_i^*,\del)\}_{i=1}^n$ of slices of $B_{X \oplus_\infty Y}$ and $\{\lambda_i\}_{i=1}^n$ with $\lambda_i \geq 0$ for all $i$ and $\sum_{i=1}^n\lambda_i=1$ such that $S:= \sum_{i=1}^n \lambda_i S_i  \subseteq B((x, y), \varepsilon)$. Then, for each $i$, by Lemma $\ref{BDP lem inf}$, we can find a slice $\widetilde{S}_i:=S(B_X,x_i^*,\alpha_i)$ of $B_{X}$ and $y_i \in B_{Y}$ such that 
$\widetilde{S}_i \times \{y_i \} \subset S_i.$
Consider $y_0=\sum_{i=1}^n \lambda_i y_i$, then 
$
\sum_{i=1}^n \lambda_i \widetilde{S}_i \times \{y_0 \} \subset \sum_{i=1}^n \lambda_i S(B_{X \oplus_\infty Y}, z_i^*, \delta_i).
$
Let $\hat{x} \in \sum_{i=1}^n \lambda_i S(B_{X}, x_i^*, \alpha_i)$. Then 
$$ 
(\hat{x}, y_0) \in \sum_{i=1}^n \lambda_i S(B_{X \oplus_\infty Y}, z_i^*, \delta_i) \subset B((x, y), \varepsilon).
$$
Thus $\|\hat{x} - x\|\leqslant \|(\hat{x}, y_0) - (x, y)\|_\infty< \varepsilon.$
So, $\hat{x} \in B(x, \varepsilon)$. Thus, $\sum_{i=1}^n \lambda_i \widetilde{S}_i \subset B(x, \varepsilon),$ showing that $x$ is a semi SCS point of $B_X$. Similarly, we can show that $y $ is a semi SCS point of $B_Y$.
\end{proof}

\begin{question}
Does the converse of Theorem \ref{inf sum semi scs} hold?
  % Let $X$ and $Y$ be Banach spaces. If $x \in$ semi SCS $B_X$ and $y \in$ semi SCS $B_Y$, does it follow that $(x, y) \in$ semi SCS $B_{X \oplus_\infty Y}$?
\end{question}

\begin{corollary} \label{inf semi SCS cor}
Let $\{X_i : i \in I\}$ be a family of Banach spaces and $X=\oplus_\infty X_i$. If $(x_i)$ is a semi SCS point of $B_X$, then $x_i$ is a semi SCS in $B_{X_i}$ for all $i \in I.$    
\end{corollary}

\begin{remark}
    %Since $\ell_\infty = \oplus_\infty \mathbb{R}$ has no semi SCS points, the converse of Corollary \ref{inf semi SCS cor} does not hold for infinite $I$.

    The converses of Corollary \ref{inf sum semi dent coro}, Corollary \ref{inf sum PC coro} and Corollary \ref{inf semi SCS cor} do not hold for infinite $I$. 
    For example, although $\mathbb{R}$ has semi denting (resp. semi PC and semi SCS) points, the space $l_{\infty}= \oplus_{\infty}\mathbb{R}$ has no such points as it satisfies the strong diameter 2 property (\cite[Proposition 3.5]{L2}). Similarly, when $K$ is a compact set without any isolated point, $B_{C(K)}$ does not have any semi denting (resp. semi PC or semi SCS) point. 
\end{remark}

%The answer for the infinite dimensional case of the above question can be answered in negative, again using the strong diameter 2 property of $\ell_\infty$.

%========================================================================================
\section{ideals of Banach spaces} \label{ideal sec}
%In this section, we discuss the stability of semi denting point (resp. semi PC, semi SCS point) in $M$-ideals. We assume that the $M$-ideals are non-trivial.
%In this section, we study lifting semi denting (resp. semi PC and semi SCS) points from M-ideals, ai-ideals and strict-ideals to the space itself. We show that the $w^*$ semi denting points and $w^*$ -semi PC points
%semi $w^*$-denting and related notions 
%can be lifted from $Y^*$ to $X^*$.  
%We begin by establishing that no point in the unit sphere of an M-ideal is a semi denting point of the closed unit ball of $X$.

\begin{proposition}
If $Y$ is an $M$-ideal in $X$, then no point of $S_Y$ is a semi denting point of $B_X$. 
\end{proposition}

\begin{proof}
   If $Y$ is an $M$-summand in $X$, then $X=Y \oplus_{\infty} Z$ for some closed subspace $Z$ of $X$. Now, if we consider $(y, 0) \in S_X,$ then by virtue of Theorem $\ref{inf sum semi dent}$, $(y, 0)$ can not be a semi denting point of $B_X$. If $Y$ is an $M$-ideal in $X,$ which is not an $M$-summand in $X,$ then $B_X$ cannot have arbitrarily small slices (\cite{BS1}).  Hence, no point of $S_Y$ can be a semi denting point of $B_X.$ Hence the result. 
\end{proof}
%=================================================

\begin{theorem} \label{mideal-semiscs}
       If $Y$ is an $M$-ideal in $X$ and $x_0^* $ is a semi $w^*$-SCS point of $B_{Y^*}$, then $x_0^*$ is a semi $w^*$-SCS point of $B_{X^*}.$
\end{theorem}
\begin{proof}
  Let $\varepsilon >0$ and $x_0^* $ be a semi $w^*$-SCS point of $B_{Y^*}$. Then there exists a convex combination of $w^*$-slices $S= \sum_{i=1}^n \lambda_i S(B_{Y^*}, y_i, \alpha_i)$ of $B_{Y^*}$, where $y_i^* \in S_{Y^*},$ $\alpha_i >0$ and $0 < \lambda_i \leq 1$ with $\sum_{i=1}^n \lambda_i =1$, for all $1 \leq i \leq n$, such that $S \subset B(x_0^*, \varepsilon)$, . Note that, since $Y$ is an $M$-ideal in $X$, $X^*=Y^* \oplus_1 Y^\perp$. So, for each $x^* \in X^*$, there is $y^* \in Y^*$ and $y^\perp \in Y^\perp$ such that $x^*=y^*+y^\perp$ and $\|x^*\|=\|y^*\|+\|y^\perp\|$. Let $0< \mu_i < \min \{\alpha_i, \varepsilon\}$ and consider the $w^*$-slices of $B_{X^*}$ given by
    \begin{align*}
       S(B_{X^*}, y_i, \mu_i) & = \{ x^* \in B_{X^*} : x^*(y_i) >1-\mu_i \} \\
        & = \{ x^* \in B_{X^*} : (y^* + y^\perp)(y_i) > 1-\mu_i \} \\
        & \subseteq S(B_{Y^*}, y_i, \alpha_i) \times \mu_i B_{Y^\perp}. 
    \end{align*}
    So, 
    \begin{align*}
        \sum_{i=1}^n \lambda_i S(B_{X^*}, y_i, \mu_i) & \subseteq \sum_{i=1}^n \lambda_i S(B_{Y^*}, y_i, \alpha_i) \times \mu_i B_{Y^\perp} \\
        & \subset B(x_0^*, \varepsilon) \times \mu_i B_{Y^\perp} \\
        & \subset B(x_0^*, 2 \varepsilon),
    \end{align*}
    which asserts that $x_0^*$ is a semi $w^*$-SCS point of $B_{X^*}$.
\end{proof}
By repeating the preceding argument for $n=1$, we obtain 
\begin{corollary}
    If $Y$ is an $M$-ideal in $X$ and $x_0^* $ is a semi $w^*$-denting point of $B_{Y^*}$, then $x_0^* $ is a semi $w^*$-denting point of $B_{X^*}.$
\end{corollary}
\begin{theorem}\label{mideal-semiPC}
	If $Y$ is an $M$-ideal in $X$ and $x_0^*$ is a semi $w^*$-PC of $B_{Y^*}$, then $x_0^*$ is a semi $w^*$-PC of $B_{X^*}.$
\end{theorem}
\begin{proof}
  	Let $\varepsilon>0$ and $x_0^* $ be a semi $w^*$-PC of $B_{Y^*}$. Then there exists a $w^*$-open subset $U$ of $B_{Y^*}$ such that $U \subset B(x_0^*, \varepsilon).$ Let $y_0^* \in S_{Y^*}\bigcap U$, $y_i\in B_Y, i=1, \cdots, n$ for some $n \in \N$ and $\delta >0$ be such that $\hat{U} := \{ y^* \in B_{Y^*} : \vert y^*(y_i) - y_0^*(y_i)\vert <\delta, 1 \leq i \leq n \}\subset U$. Then $\hat{U}$ is a $w^*$-open set in $B_{Y^*}$ and as $Y$ is an $M$-ideal in $X$, by following the techniques in \cite[Proposition 2.6]{BS1}, we find that $\hat{U} \subset U + \varepsilon B_{Y^\perp}.$ Thus, $ \hat{U} \subset B(x_0^*, \varepsilon)+ \varepsilon B_{Y^\perp} \subset B(x_0^*, 2 \varepsilon).$ This results that $x_0^* $ is a semi $w^*$-PC of $B_{X^*}.$
\end{proof}

\begin{theorem}
    Let $Y$ be a strict ideal of $X.$ If $x_0^* $ is a semi $w^*$-SCS point of $B_{Y^*}$, then $x_0^* $ is a semi $w^*$-SCS point of $B_{X^*}.$ 
\end{theorem}
\begin{proof}
    Let $\varepsilon >0$ and $x_0^*$ is a semi $w^*$-SCS point of $B_{Y^*}$. Then there exists a convex combination of $w^*$-slices $S= \sum_{i=1}^n \lambda_i S_i$, $S_i=S(B_{Y^*}, y_i, \alpha_i)$ of $B_{Y^*}$ such that $S \subset B(x_0^*, \varepsilon)$, where $y_i^* \in S_{Y^*},$ for all $i=1,2, \cdots n$, $\alpha_i >0$, $0 < \lambda_i \leq 1$ with $\sum_{i=1}^n \lambda_i =1$. Since $Y$ is a strict ideal in $X$, we have $B_{X^*}=\overline{B_{Y^*}}^{w^*}.$ Therefore,
    \begin{align}
        \overline{S_i} & = \{x^* \in B_{X^*} : x^*(y_i) > 1-\alpha_i \} \\
        & = \{x^* \in \overline{B_{Y^*}}^{w^*}: x^*(y_i) >1-\alpha_i \}.
    \end{align}
    Thus, $\sum_{i=1}^n \lambda_i \overline{S_i} \subseteq \sum_{i=1}^n \lambda_i S_i \subset B(x_0^*, \varepsilon).$ This proves that $x_0^*$ is a semi $w^*$-SCS point of $B_{X^*}$.
\end{proof}
By repeating the preceding argument for $n=1$, we obtain 
\begin{corollary}
    Let $Y$ be a strict ideal of $X.$ If $x_0^* $ is a semi $w^*$-denting point of $B_{Y^*}$, then $x_0^*$ is a semi $w^*$-denting point of $B_{X^*}.$
\end{corollary}

\begin{theorem}
	Let $Y$ be a strict ideal of $X.$ If $x_0^* $ is a semi $w^*$-PC of $B_{Y^*}$, then $x_0^* $ is a semi $w^*$-PC of $B_{X^*}.$ 
\end{theorem}

\begin{proof}
	Let $\varepsilon > 0$ and $x_0^* $ is a semi $w^*$-PC of $B_{Y^*}$. Then there exists a relatively $w^*$-open subset, say $U=\{y^* \in B_{Y^*} : \vert y^*(y_i) - y_0^*(y_i)\vert < \delta , 1\leq i \leq n\}$ of $B_{Y^*}$ such that $U \subset B(x_0^*, \varepsilon),$ where $y_i \in B_Y$, $1 \leq i \leq n$.
	Consider a relatively $w^*$-open subset $\hat{U}= \{x^* \in B_{X^*} : \vert x^*(y_i) - y_0^*(y_i)\vert < \delta , 1 \leq i \leq n\}$ of $B_{X^*}$.
	  Following the arguments as in \cite[Proposition 2.10]{BS1}, we find that $\hat{U} \subset \overline{U}^{w^*} \subset \overline{B(x_0^*, \varepsilon)}^{w^*}$. Now, let us take $x^* \in \hat{U}$, then $x^* \in \overline{B(x_0^*, \varepsilon)}^{w^*}$. i.e., there exists a net $(x_\lambda^*)$ in $B(x_0^*, \varepsilon)$ such that $x_\lambda^* \xrightarrow{w^*} x^*$. This further implies that $x_\lambda^* - a^* \xrightarrow{w^*} x^* - x_0^*$. Finally, we get that
    $$
    \Vert x^* - x_0^* \Vert \leq \liminf_{\lambda} \Vert x_\lambda^* - x_0^* \Vert < \varepsilon.
    $$
    So, $\Vert x^* - x_0^* \Vert \leq \varepsilon < 2 \varepsilon,$ showing that $x^* \in B(x_0^*, 2 \varepsilon)$. Thus, $\hat{U} \subset B(x_0^*, 2\varepsilon).$ This proves that $x_0^*$ is a semi $w^*$-PC of $B_{X^*}$. 	
\end{proof} 

The following proposition is crucial for the upcoming theorem.
\begin{proposition} \label{ai}
\cite{ALN}  A subspace $Y$ of $X$ is an $ai$-ideal if and only if there exists a Hahn-Banach extension operator $f: Y^* \rightarrow  X^*$ such that for every $\varepsilon > 0$ and every finite dimensional subspace $F\subset X,$ finite dimensional subspace $F^*\subset Y^*$, there exists $T: F \rightarrow Y$ such that the following conditions hold:
\begin{enumerate}
\item $Tu = u $, $\forall u \in F\bigcap Y;$
\item $\frac{1}{1+\varepsilon} \Vert u \Vert \leqslant \Vert Tu \Vert \leqslant (1+\varepsilon) \Vert u \Vert$, $\forall u \in F;$
\item $fy^*(u) = y^*(Tu)$, $\forall u \in F$ and $y^* \in F^*.$  
\end{enumerate}  
\end{proposition}

\begin{theorem}\label{semi SCS ai ideal}
If $Y$ is an $ai$-ideal in $X$ and $x_0 $ is a semi SCS point of $B_Y$, then $x_0 $ is a semi SCS point of $B_X$.
\end{theorem}
\begin{proof}
Let $\varepsilon >0$ and $x_0$ be a semi SCS point of $B_Y$. Then there exists a convex combination of slices, $S=\sum_{i=1}^{n} \lambda_i S(B_Y,y_i^*,\alpha_i)$ of $B_Y$ such that $S \subset B(x_0, \varepsilon),$ where  $i=1,2,...,n$, $y_i^*\in S_{Y^*}$, $\alpha_i > 0$, $1 \geq \lambda_i > 0$ with $\sum_{i=1}^{n} \lambda_i =1.$
Since $Y$ is an $ai$-ideal in $X,$ we can guarantee a Hahn Banach extension operator, $f: Y^* \rightarrow  X^*$ satisfying the conditions of Proposition $\ref{ai}$.
Consider a convex combination of slices $\hat{S}=\sum_{i=1}^{n} \lambda_i S(B_X,fy_i^*,\alpha_i)$ of $B_X$ and $x=\sum_{i=1}^{n} \lambda_i x_i \in \hat{S}.$
Now we can choose $0<\alpha<1$  such that $\frac{x_i}{1+\alpha} \in S(B_X,fy_i^*,\alpha_i)$, $\forall i=1,2,...,n.$
Put $x'=\sum_{i=1}^{n} \lambda_i \frac{x_i}{1+\alpha}=\frac{x}{1+\alpha} \in \hat{S}.$
Put $F= {\rm span} \{\frac{x_1}{1+\alpha},\frac{x_2}{1+\alpha},...\frac{x_n}{1+\alpha}, x_0 \}$ and $F^*= {\rm span} \{y_1^*,y_2^*,...y_n^*\}.$
For $F,F^*$ and $\alpha$, there exists a mapping $T:F\rightarrow Y$ satisfying the conditions in Proposition $\ref{ai}.$ %Now, $Tx',T\tilde{x}'\in Y.$
Since $x' \in F$, $\Vert Tx'\Vert\leqslant (1+\alpha)\Vert x'\Vert\leqslant 1.$ \\
Also, $\forall i=1,2,...,n,$ by Proposition $\ref{ai}$ (iii),
$$ y_i^* \bigg (T\frac{x_i}{1+\alpha} \bigg )=fy_i^* \bigg (\frac{x_i}{1+\alpha} \bigg )>1-\alpha_i,$$
which implies $Tx'\in S.$ This implies $\Vert T x' - x_0 \Vert < \varepsilon.$ Now,
\begin{align*} 
    \Vert x' - x_0 \Vert & \leq (1+\alpha) \Vert T x' - x_0 \Vert \\
    & \leq (1+\alpha) \Vert T(x'- x_0) \Vert \\
    &=(1+\alpha) \Vert T x'- T x_0 \Vert \\
    & < (1+\alpha) \varepsilon.
\end{align*}
Finally, 
\begin{align*}
    \Vert x- x_0 \Vert & \leq \Vert x'- x_0 \Vert + \alpha  < (1+\alpha) \varepsilon + \alpha = \varepsilon^{'} ~ \textit{(say)}.
\end{align*}   
This shows that $x \in B(x_0, \varepsilon),$ which proves that $x_0 $ is a semi SCS point of $B_X$.
\end{proof}
By repeating the preceding argument for $n=1$, we obtain 
\begin{corollary}\label{bdp ai ideal}
If $Y$ is an $ai$-ideal in $X$ and $ x_0$ is a semi denting point of $B_Y$, then $x_0 $ is a semi denting point of $B_X$.
\end{corollary} 
\begin{theorem} \label{bhp ai ideal}
If $Y$ is an $ai$-ideal in $X$ and $x_0 $ is a semi PC of $B_Y$, then $x_0 $ is a semi PC of $B_X$.
\end{theorem}
\begin{proof}
Let $\varepsilon > 0$ and $x_0 $ is a semi PC of $B_Y$. Then there exists a relatively weakly open subset $U=\{y\in B_Y : \vert y_i^*(y-y_0)\vert <\delta , 1 \leq i \leq n\}$, of $B_Y$ such that $U \subset B(x_0, \varepsilon),$ where $y_0\in B_Y$, $\delta> 0$, $y_i^*\in Y^*$, $\forall i=1,2,...,n$. Since $Y$ is an $ai$-ideal in $X,$ we can guarantee a Hahn-Banach extension operator $f : Y^* \rightarrow  X^*$ as in Proposition $\ref{ai}$.
Consider a nonempty relatively weakly open subset $\hat{U}=\{x\in B_X : \vert fy_i^*(x-y_0)\vert <\delta , 1 \leq i \leq n\}$ of $B_X$ and let us take $x \in \hat{U}$.
%Here we use a similar technique as in \cite{ALN}.
 Choose  $0<\alpha<1$  such that 
$\tilde{x} = \frac{x}{1+\alpha} \in \hat{U}.$ Put $F= {\rm span} \{y_0,\tilde{x}, x_0 \}$ and $F^*= {\rm span} \{y_1^*,y_2^*,...y_n^*\}.$ 
For $F,F^*$ and $\alpha$ there exists $T:F\rightarrow Y$ as in Proposition $\ref{ai}.$ %and so $T\tilde{x_},T\tilde{y}\in Y.$
By  Proposition $\ref{ai}$ (ii), $\Vert T\tilde{x}\Vert\leqslant (1+\alpha)\Vert \tilde{x}\Vert\leqslant 1.$ \\
Also by Proposition $\ref{ai}$ (i) and (iii), we have for all $i=1,2,...,n,$
$$
\vert y_i^*(T\tilde{x} - y_0) \vert = \vert y_i^*(T\tilde{x} - Ty_0)\vert = \vert fy_i^*(\tilde{x}-y_0)\vert<\delta .
$$
Thus, $T\tilde{x}\in U \subset B(x_0, \varepsilon).$ Following the same arguments as in Theorem \ref{semi SCS ai ideal}, we can prove that $x_0 $ semi PC of $B_X$.
\end{proof}
%====================================================================
%projective tensor product
\section{projective tensor products of Banach Spaces} \label{tensor}

%In this section, we investigate conditions that characterize semi denting and semi SCS points in the unit ball of the projective tensor product $ \ X\widehat {\otimes}_{\pi} Y$.

\begin{theorem}\label{projective semi -scs}
 If $a$ is a semi denting point of $B_{X}$ and $b $ is a semi SCS point of $B_{Y}$, then $a \otimes b$ is a semi SCS point of $B_{X \widehat{\otimes}_{\pi} Y }$.
\end{theorem}
\vspace{-0.19in}

\begin{proof}
    Let $\varepsilon >0$, $a $ be a semi denting point of $B_{X}$ and $b $ be a semi SCS point of $B_{Y}$. Then there exists a slice $S(B_{X}, x^*, \delta_1)$ of $B_{X}$ such that
    $$
    S(B_{X}, x^*, \delta_1) \subset B(a, \varepsilon).
    $$
    Also, there exists a collection $\{S(B_{Y}, y_i^*, \delta_{2, i})\}_{i=1}^n$ of slices of $B_{Y}$ and $\{\lambda_i\}_{i=1}^n \subseteq \R$, for some $n\in \N$, such that $\lambda_i \geq 0$ for all $i$ and $\sum_{i=1}^n\lambda_i=1$ satisfying
    $$ \sum_{i=1}^n \lambda_i S(B_{Y}, y_i^*, \delta_{2, i}) \subset B(b, \varepsilon).$$
    
    For each $1 \leq i \leq n$, consider the slice $T_i=S(B_{X \widehat{\otimes}_{\pi} Y}, x^* \otimes y_i^*, \delta^2)$ of $B_{X \widehat{\otimes}_{\pi} Y}$, where $0 < 2 \delta < \min \{ \varepsilon, \delta_1, \delta_{2, i} \}_{i=1}^n.$ Let $u_i \in T_i \cap \text{co} (B_{X} \otimes B_{Y})$. Then $u_i=\sum_{j=1}^{m_i} \lambda_{i, j} (v_{i,j} \otimes w_{i,j})$, where $v_{i,j} \in B_{X}$, $w_{i,j} \in B_{Y}$, $\lambda_{i,j} \geq 0$ for all $1 \leq j \leq m_i$ and $\sum_{i=1}^{m_i} \lambda_{i,j}=1$. Denote 
    $$
    Q_i=\{ 1 \leq j \leq m_i : x^*(v_{i,j}) y_i^*(w_{i,j}) > 1-\delta \},
    $$
    $$
    Q_i^c= \{1, \cdots, m_i \}\setminus Q_i.
    $$
    Then 
    \begin{align*}
        1 - \delta^2 & < (x^* \otimes y_i^* ) (u_i) \\
        & = \sum_{j \in Q_i} \lambda_{i, j} x^*(v_{i,j}) y_i^* (w_{i,j}) + \sum_{j \in Q_i^c} \lambda_{i, j} x^*(v_{i,j}) y_i^* (w_{i,j}) \\
        & \leq \sum_{j \in Q_i} \lambda_{i, j} + (1-\delta) \sum_{j \in Q_i^c} \lambda_{i, j} \\
        & = 1-\sum_{j \in Q_i^c} \lambda_{i, j} + (1-\delta) \sum_{j \in Q_i^c} \lambda_{i, j} \\
        & = 1-\delta \sum_{j \in Q_i^c} \lambda_{i, j}.
    \end{align*}
    This implies that $\sum_{j \in Q_i^c} \lambda_{i, j} < \delta$ and $\sum_{j \in Q_i} \lambda_{i, j} > 1- \delta$. \\
    On the other hand
    \begin{align*}
        y_i^* \bigg ( \sum_{j \in Q_i} \lambda_{i, j} w_{i,j} \bigg ) & \geq (x^* \otimes y_i^* ) ( \sum_{j \in Q_i} \lambda_{i, j} v_{i,j} \otimes w_{i,j} ) \\
        & = \sum_{j \in Q_i} \lambda_{i, j} x^*(v_{i,j}) y_i^*(w_{i,j} ) \\
        & > (1-\delta)(1-\delta) \\
        & > 1 - 2 \delta > 1-\delta_{2, i}.
    \end{align*}
  So, $ \sum_{j \in Q_i} \lambda_{i, j} w_{i,j} \in S(B_{Y}, y_i^*, \delta_{2, i}) $. Therefore, we have $\sum_{i=1}^n \lambda_i \sum_{j \in Q_i} \lambda_{i, j} w_{i,j} \subseteq B(b, \varepsilon)$.
  %\subset B(b, \varepsilon).$ 
  Similarly, we have 
  \begin{align*}
      x^*(v_{i,j}) & \geq (x^* \otimes y_i^* ) ( \sum_{j \in Q_i} \lambda_{i, j} v_{i,j} \otimes w_{i,j} ) \\
        & = \sum_{j\in Q_i} \lambda_{i, j} x^*(v_{i,j}) y_i^*(w_{i,j} ) \\
      & >(1-\delta)^2 > 1-\delta_1,
  \end{align*} 
  which shows that $v_{i,j} \in S(B_{X}, x^*, \delta_1) \subset B(a, \varepsilon).$ Finally,
  \begin{align*}
      \Vert \sum_{i=1}^n \lambda_i u_i - a \otimes b  \Vert & \leq \Vert \sum_{i=1}^n \lambda_i u_i - \sum_{i=1}^n \lambda_i ( a \otimes \sum_{j \in Q_i} \lambda_{i, j} w_{i,j} ) \Vert \\
      & + \Vert \sum_{i=1}^n \lambda_i ( a \otimes \sum_{j \in Q_i} \lambda_{i, j} w_{i,j} ) - a \otimes b  \Vert \\
     & \leq \sum_{i=1}^n \Vert u_i - a \otimes \sum_{j \in Q_i} \lambda_{i, j} w_{i,j} \Vert \\
     & + \Vert \sum_{i=1}^n \lambda_i ( a \otimes \sum_{j \in Q_i} \lambda_{i, j} w_{i,j} ) - a \otimes b  \Vert .
  \end{align*}
  Let us now compute
  \begin{align*}
      \Vert u_i - a \otimes \sum_{j \in Q_i} \lambda_{i, j} w_{i,j} \Vert & \leq \Vert \sum_{j \in Q_i} \lambda_{i, j} v_{i,j} \otimes w_{i,j} - a \otimes \sum_{j \in Q_i} \lambda_{i, j} w_{i,j} \Vert + \Vert \sum_{j \in Q_i^c} \lambda_{i, j} v_{i,j} \otimes w_{i,j} \Vert \\
      & \leq \sum_{j \in Q_i} \lambda_{i, j} \Vert (v_{i,j} - a ) \otimes w_{i,j} \Vert + \sum_{j \in Q_i^c} \lambda_{i, j} \Vert v_{i,j} \Vert \Vert w_{i,j} \Vert \\
      & < \sum_{j \in Q_i} \lambda_{i, j} \varepsilon + \sum_{j \in Q_i^c} \lambda_{i, j} \\
      & < \varepsilon + \delta < 2 \varepsilon ,
  \end{align*}
  and
  \begin{align*}
      \Vert \sum_{i=1}^n \lambda_i ( a \otimes \sum_{j \in Q_i} \lambda_{i, j} w_{i,j} ) - a \otimes b  \Vert & = \Vert a \otimes \bigg (\sum_{i=1}^n \lambda_i \sum_{j \in Q_i} \lambda_{i, j} w_{i,j} - b \bigg ) \Vert \\
      & < \varepsilon .
  \end{align*}
  Therefore,
\[
      \Vert \sum_{i=1}^n \lambda_i u_i - a \otimes b  \Vert < 2 \varepsilon + \varepsilon = 3 \varepsilon.
      \]

  Hence, $\sum_{i=1}^n \lambda_i T_i \subset B(a \otimes b, 3 \varepsilon),$ which proves that $a \otimes b $ is a semi SCS point of $B_{X \widehat{\otimes}_{\pi} Y }$.
    \end{proof}

By repetition of the preceding argument for $n=1$, we obtain 
    \begin{corollary}\label{projective semi -denting}
%        Let $X$ and $Y$ be Banach spaces. 
If $a $ is a semi denting point of $B_{X}$ and $b $ is a semi denting point of $B_{Y}$, then $a \otimes b $ is a semi denting point of $B_{X \widehat{\otimes}_{\pi} Y }$.
    \end{corollary}

 \begin{question}
     Do the converses of Theorem \ref {projective semi -scs} and Corollary \ref {projective semi -denting} hold?
 \end{question}

\noindent\textbf{{Conflicts of Interest}} \ \ \ \ The authors declare that they have no conflict of interest.

\noindent\textbf{Funding} \ \ \ \ Not applicable.

   \begin{Acknowledgement}
This work was done when the first and second authors were visiting the National Institute of Science Education and Research (NISER), Bhubaneswar, India. They are grateful to Professor Anil Karn, Department of Mathematical Sciences, NISER, for his hospitality and support. 
The first author is grateful to Dr. Timothy Clarke, Chair, Department of Mathematics and Statistics, Loyola University, for his support and encouragement. She is also grateful to Professor Bahram Roughani, Associate Dean of the College of Arts and Sciences, Loyola University, for providing her with the Dean's supplemental grant during her travel in Summer 2025. 
The second and the third authors thank Professor C. Nahak, Department of Mathematics, IIT Kharagpur, for his support. 
The research of the third author is supported by the Institute Postdoctoral Fellowship, IIT Kharagpur.
The research of the fourth author is supported by the Institute Postdoctoral Fellowship, NISER Bhubaneswar.
   \end{Acknowledgement}

\enlargethispage{4\baselineskip}

\end{document}